\documentclass[10pt]{amsart}
\usepackage{amsxtra, amsfonts, amsmath, amsthm, amstext, amssymb, amscd, mathrsfs, verbatim, color}
\usepackage{threeparttable}
\usepackage{mathtools}
\usepackage[ansinew]{inputenc}\usepackage[T1]{fontenc}
\usepackage[all,cmtip]{xy}
\theoremstyle{plain}

\newtheorem{theorem}{Theorem}[section]
\newtheorem{lemma}[theorem]{Lemma}
\newtheorem{definition-theorem}[theorem]{Definition-Theorem}
\newtheorem{proposition}[theorem]{Proposition}
\newtheorem{corollary}[theorem]{Corollary}

\theoremstyle{definition}
\newtheorem{definition}[theorem]{Definition}
\newtheorem{example}[theorem]{Example}
\newtheorem{remark}[theorem]{Remark}
\newtheorem{notation}[theorem]{Notation}

\newcommand \bth[1] { \begin{theorem}\label{t#1} }
\newcommand \ble[1] { \begin{lemma}\label{l#1} }

\newcommand \bpr[1] { \begin{proposition}\label{p#1} }
\newcommand \bco[1] { \begin{corollary}\label{c#1} }
\newcommand \bde[1] { \begin{definition}\label{d#1}\rm }
\newcommand \bex[1] { \begin{example}\label{e#1}\rm }
\newcommand \bre[1] { \begin{remark}\label{r#1}\rm }

\newcommand \bnota[1] {\begin{notation}\label{n#1}\rm }

\newcommand {\ele} { \end{lemma} }

\newcommand {\epr} { \end{proposition} }
\newcommand {\eco} { \end{corollary} }
\newcommand {\ede} { \end{definition} }
\newcommand {\eex} { \end{example} }
\newcommand {\ere} { \end{remark} }
\newcommand {\enota} { \end{notation} }

\def \Id { {\mathrm{Id}} }

\DeclareMathOperator \Hom { {\mathrm{Hom}} }

\DeclareMathOperator \im { { {\mathrm im}}}

\begin{document}
\setlength{\baselineskip}{1.2\baselineskip}
\title[First extensions]
{Some methods of computing first extensions between modules of graded Hecke algebras}
\author[Kei Yuen Chan]{Kei Yuen Chan}
\address{ Korteweg-de Vries Institute for Mathematics, Universiteit van Amsterdam}
\email{K.Y.Chan@uva.nl, keiyuen.chan@gmail.com}

\maketitle 
\begin{abstract}
In this paper, we establish connections between the first extensions of simple modules and certain filtrations of of standard modules in the setting of graded Hecke algebras. The filtrations involved are radical filtrations and Jantzen filtrations. Our approach involves the use of information from the Langlands classification as well as some deeper understanding on some structure of some modules. Such module arises from the image of a Knapp-Stein type intertwining operator and is a quotient of a generalized standard module. Along the way, we also deduce some results on the blocks for finite-dimensional modules of graded Hecke alebras.

As an application, we compute the Ext-groups for irreducible modules in a block for the graded Hecke algebra of type $C_3$, assuming the truth of a version of Jantzen conjecture. 


\end{abstract}

\section{Introduction} \label{s intro}

\subsection{}
Let $R$ be a reduced root system and let $\Pi$ be a fixed set of simple roots in $R$. Let $\mathbb{H}$ be the Lusztig's graded (affine) Hecke algebra associated to a root datum $(R,V, R^{\vee}, V^{\vee}, \Pi)$ and a parameter function $k$ (see Definition \ref{def graded affine}). This paper continues our study \cite{Ch, Ch2} of the $\mathrm{Ext}$-groups in the category of $\mathbb{H}$-modules. The representation theory of graded Hecke algebras can be transferred to the counterpart of affine Hecke algebras by Lusztig's reduction theorems \cite{Lu} and hence is also useful in understanding the representation theory of $p$-adic groups. 

The graded Hecke algebra is a deformation of a skew group ring. The category of representations of a graded Hecke algebra however does not possess a (natural) tensor product, which differs from the one of a skew group ring. An $\mathbb{H}$-module still admits a Koszul type resolution analogous to the one for skew group ring, which allows one to establish a Poincar\'e duality on $\mathrm{Ext}$-groups and establish a formula for the Euler-Poincar\'e pairing depending only on the reflection group structure of $\mathbb{H}$-modules \cite{Ch2}. 

The extensions between a discrete series and a tempered module have been determined independently in \cite{Me}, \cite{OS} and \cite{Ch2} (for different but related settings). Then it is natural to consider modules outside those pairs. One possible direction is to study extensions between tempered modules in the setting of graded Hecke algebras analogous to the results of Opdam-Solleveld \cite{OS2}. Another possible direction is to consider extensions for some non-tempered modules. Our study mainly arises when we consider extensions of the latter pairs.

The main result (Theorem \ref{thm first ext sum}) in this paper is to establish connections between first extensions of simple $\mathbb{H}$-modules and certain filtrations of standard modules. Such result is then applied to compute some examples at the end. A novelty of our approach is to use an uniqueness property for a type of module which we call generalized standard modules.

Jantzen filtrations and radical filtrations are involved in our study. It is interesting to see if various filtrations are compatible. A number of such study for Verma modules as well as Weyl modules can be found in the literature, and are closely related to the Kazhdan-Lusztig theory and an interpretation of Kazhdan-Lusztig polynomials by Vogan \cite{Vo}. In our case of graded Hecke algebras, radical filtrations and Jantzen filtrations in general do not coincide (see Example \ref{ss remark filt}). In that example, the socle filtration is also not the same as the Jantzen filtration. It would be an interesting question to see if there are some deeper relations between those filtrations and with related geometry.

Our approach of the study makes use of an indecomposable module, which arises from the image of an intertwining operator and is a quotient of an indecomposable generalized standard module. The structure of such module in terms of composition factors can be described in some details (Theorem \ref{thm JF general stand}) and hence provides a platform for transferring information between Jantzen filtrations and $\mathrm{Ext}$-groups.

\subsection{}

In order to describe our results in more detail, we need more notations. We first recall the Langlands classification for $\mathbb{H}$. For each $J \subseteq \Pi$, one can associate a parabolic subalgebra $\mathbb{H}_J$ of $\mathbb{H}$ (see Notation \ref{not parabolic}). The Langlands classification \cite{Ev} states that simple $\mathbb{H}$-modules can be parametrized by the set $\Xi_L$ of pairs $(J,U)$, where $J \subseteq \Pi$ and $U$ is certain irreducible $\mathbb{H}_J$-module (analogous to the tempered representation twisted by a character for reductive groups). 

The more explicit construction of simple modules from those Langlands classification data $(J,U) $ is as follows. Given a pair $(J,U)\in \Xi_L$, one constructs a parabolically induced modules $I(J,U) :=\mathbb{H} \otimes_{\mathbb{H}_J} U$, which is usually referred to standard modules in the literature. The Langlands classification asserts that $I(J,U)$ has a unique simple quotient, denoted by $L(J,U)$, and all simple modules arises from this way.

We first state a useful vanishing result on $\mathrm{Ext}$-groups between some standard modules and some simple modules. For each $(J,U) \in \Xi_L$, we associate an element $\nu(J,U) \in V^{\vee}$ (see Definition \ref{eqn temp}). Recall that there is a partial ordering $\leq$ on the set $\Xi_L$ determined by the dominance ordering on the elements $\nu(J,U) \in V^{\vee}$.
\begin{proposition} \label{main thm 1} (Proposition \ref{prop ext lower langlands para})
Let $(J_1,U_1), (J_2, U_2) \in \Xi_L$. If $\nu(J_1, U_1) \not\leq \nu(J_2, U_2)$, then 
\[ \mathrm{Ext}^i_{\mathbb{H}}(I(J_1, U_1), L(J_2, U_2)) =\mathrm{Ext}^i_{\mathbb{H}}(I(J_1, U_1), I(J_2, U_2))=0 .\]
\end{proposition}
\noindent
Proposition \ref{main thm 1} is obtained by comparing the Langlands classification parameters and standard homological algebras. One may compare with methods for the highest weight representations of complex semisimple Lie algebras.



Using Proposition \ref{main thm 1} with a certain contravariant operator on $\mathbb{H}$-modules, one can relate most (but not all) of the first extension of two non-isomorphic simple modules in terms of the second layer of the radical filtration of a standard module. The problem for understanding the first self-extension of a simple module is much harder and has less clue from the standard information of Langlands classification.

For determining self-extensions of simple modules, our approach of using standard modules logically leads to study a structure of generalized standard modules. Here a generalized standard module is an indecomposable module which admits a certain filtration by standard modules (see Definition \ref{def generalized St} and Proposition \ref{prop indecomp standard} for the details).

An approach to understand such structure is to use an intertwining operator between a generalized standard module and a certain dual of a generalized standard module. The intertwining operator is normalized in a way to have properties analogous to a Knapp-Stein intertwining operator for real reductive groups \cite{ KS} (see Proposition \ref{prop intertwin properties}). One may also compare with other study of intertwining operators \cite{Ro}, \cite{Re0}, \cite{KR}, \cite{DO2} and\cite{So}. The study has some common features, but we also prove some results not appearing in the literature. 

Such intertwining operator defines a Jantzen type filtration and in particular defines an interesting quotient from the first layer of the Jantzen filtration. In the case of standard modules, the first layer of Jantzen filtration is simply the simple quotient and so we may regard such quotient to be a generalization of the simple quotient (see Definition \ref{def quotient gen}). A main technical result (Theorem \ref{thm JF general stand}) in this paper is to describe the structure of such quotient for special cases in terms of the Jantzen filtration of the (ordinary) standard module, and characterize such module by certain property with respect to the composition factors. This is also the pathway to connect some first self-extensions of simple modules and the Jantzen filtration. 

The special cases of generalized standard modules considered in Theorem \ref{thm JF general stand} are called of $S(V)$-type, meaning that extensions are arising from extensions of representations of a polynomial algebra. Then one can exploit the well-known extensions for polynomials algebras.

\subsection{} \label{ss intro 3} To describe another result, we need more notations. Let $(J,U) \in \Xi_L$. Let $V_J$ be the space spanned by simple roots in $J$. Let $V_J^{\vee, \bot}$ be the subspace of $V^{\vee}$ containing functionals vanishing on $V_J$ (Notation \ref{not parabolic}). For each $\eta^{\vee} \in V_J^{\vee, \bot}$ (possibly zero for our definition), we can deform the induced module $\mathbb{H} \otimes_{\mathbb{H}_J}U$ along the direction $\eta^{\vee}$ to obtain a family of $\mathbb{H}$-modules $\mathbb{H} \otimes_{\mathbb{H}_J} U_{\mathbf t\eta^{\vee}}$, where $\mathbf t$ is an indeterminate. One can associate an intertwining operator $\Delta_{\mathbf t \eta^{\vee}}$ for $\mathbb{H} \otimes_{\mathbb{H}_J} U_{\mathbf t\eta^{\vee}}$ and the vanishing order of elements under $\Delta_{\mathbf t\eta{\vee}}$  defines a Jantzen-type filtration, denoted $\mathrm{JF}^i_{\eta^{\vee}}(J,U)$ on the standard module $\mathbb{H} \otimes_{\mathbb{H}_J} U$. For the details of the notations, see Section \ref{s jan filt}. As discussed before, the Jantzen type filtration is also defined for generalized standard modules in Section \ref{s jan filt} as the intertwining operator $\Delta_{\mathbf t\eta^{\vee}}$ is defined for all generalized standard modules.

We define 
\[ V^{\bot}_{\mathrm{bad}}(J,U) = \left\{ \eta^{\vee} \in V_J^{\vee, \bot} : \mathrm{JF}_{\eta^{\vee}}^1(J,U)=\mathrm{JF}_{\eta^{\vee}}^2(J,U) \right\} ,\]
which forms a vector space. We now give a version of Theorem \ref{thm jantan differential}.
\begin{theorem} (part of Theorem \ref{thm jantan differential}) \label{main thm 2}
Let $(J,U) \in \Xi_L$. Let $\mathbb{H}_J \cong \mathbb{H}_J^{ss} \otimes S(V_J^{\bot})$ be a decomposition as in Notation \ref{not parabolic}.  Suppose $\mathrm{Res}_{\mathbb{H}_J^{ss}}U$ is a discrete series, or more generally 
\[\mathrm{Ext}^1_{\mathbb{H}^{ss}_J}(\mathrm{Res}_{\mathbb{H}^{ss}_J} U, \mathrm{Res}_{\mathbb{H}^{ss}_J} U)=0 .\] 
Here $\mathrm{Res}_{\mathbb{H}^{ss}_J}$ is the restriction functor to $\mathbb{H}^{ss}_J$-module. Then 
\[ \mathrm{Ext}^1_{\mathbb H}(L(J,U), L(J,U)) \cong V^{\bot}_{\mathrm{bad}}(J,U) .\]
\end{theorem}


A novelty of the proof for Theorem \ref{main thm 2} is essentially using a uniqueness property.

Using Theorem \ref{main thm 1} and Theorem \ref{main thm 2}, the first extensions of simple modules can be much determined by information from filtrations of standard modules. We state a version of Theorem \ref{thm first ext sum}:

\begin{theorem} (Theorem \ref{thm first ext sum}) \label{main result 3}
Let $(J_1, U_1), (J_2, U_2) \in \Xi_L$. 
\begin{enumerate}
\item Suppose $\nu(J_1, U_1) \neq \nu(J_2, U_2)$. If $\mathrm{Ext}^1_{\mathbb{H}}(L(J_1, U_1), L(J_2, U_2)) \neq 0$, then either $L(J_1, U_1)$ is isomorphic to a (simple) subquotient of $I(L_2, U_2)$ or $L(J_2, U_2)$ is isomorphic to a (simple) subquotient of $I(J_1, U_1)$.
\item Suppose $\nu(J_1, U_1) =\nu(J_2, U_2)$ (and in particular $J_1=J_2$). Suppose  
\[ \mathrm{Ext}^1_{\mathbb{H}^{ss}_{J_1}}(\mathrm{Res}_{\mathbb{H}^{ss}_{J_1}} U_1, \mathrm{Res}_{\mathbb{H}^{ss}_{J_1}} U_2)=0 . \] 
Then $\mathrm{Ext}^1_{\mathbb{H}}(L(J_1, U_1), L(J_2, U_2))$ is determined by the second layer of Jantzen filtrations in the sense of Theorem \ref{main thm 2} when $U_1 \cong U_2$ and is zero when $U_1 \not\cong U_2$.
\end{enumerate}
\end{theorem}
As an application of our results, we compute $\mathrm{Ext}$-groups between some simple modules based on some computations of composition factors for standard modules in \cite{Ci} at the end. We remark that results in this paper are independent of \cite{Ch2}, but we shall need some results in \cite{Ch2} for computing some examples.


\subsection{} It leaves some questions from our study. First, we do not consider non-$S(V)$-type extensions, which roughly means dropping the hypothesis in Theorem \ref{main thm 2}. Second, finding an effective way to compute radical filtration and Jantzen filtration in general is still an open problem. For the second question, there is a conjectural way to compute Jantzen filtration by geometric means (see e.g. \cite{BC}) using Kazhdan-Lusztig polynomials \cite{Lu2, Lu3} (also see \cite{CG}, \cite{Lu0}) and moreover the Arakawa-Suzuki functor \cite{AS, Su} as well as results of Rogawski \cite{Ro}  determine some Jantzen filtrations in type $A_n$ case. This sheds some light on understanding first self-extensions in such directions. 

\subsection{} We give an organization of this paper. Section \ref{s intro} is the introduction. Section \ref{s prelim} recalls basic definitions and states some basic results. Section \ref{s gss ext} defines the generalized standard modules and deduces several results about indecomposability and extensions from the Langlands classification. Section \ref{s intertwin operat} studies a certain normalized intertwining operator for generalized standard modules and show it has properties analogous to a Knapp-Stein intertwining operator. Section \ref{s jan filt} studies the Jantzen filtration of a generalized standard modules, which is defined from the intertwining operation in Section \ref{s intertwin operat}. Such intertwining operator defines a certain quotient on the generalized standard modules. Section \ref{s quotient gsm} specifies on certain generalized standard modules and we describe the structure of the quotient in terms of the Jantzen filtration of (ordinary) standard modules. Section \ref{s first ext} summarizes our study by giving a connection between some first extensions of simple modules and filtrations on standard modules. Theorem \ref{thm first ext sum} is our main result. 


\subsection{Notation} For an algebra $\mathbb{A}$ and an $\mathbb{A}$-module $X$, we write $\pi_X(a)x$ or $a.x$ or $ax$ for the action of $a$ on $x \in X$.




\subsection{Acknowledgment} The author would like to thank Dan Ciubotaru and Peter Trapa for discussions on Jantzen filtrations and would like to thank Eric Opdam for discussions on intertwining operators. This research was supported by both of the ERC-advanced grant no. 268105 from Eric Opdam and the Croucher Fellowship. 

\section{Preliminaries} \label{s prelim}
\subsection{Root systems} \label{ss basic notation}
Let $R$ be a reduced root system. Let $\Pi$ be a fixed choice of simple roots in $R$. Then $\Pi$ determines the set $R^+$ of positive roots. For $\alpha \in R^+$, sometimes write $\alpha>0$. For $\alpha \in R \setminus R^+$, write $\alpha <0$. Let $W$ be the finite reflection group of $R$. Let $V_0$ be a real vector space spanned by $R$. For any $\alpha \in \Pi$, let $s_{\alpha}$ be the simple reflection in $W$ associated to $\alpha$ (i.e. $\alpha \in V_0$ is in the $(-1)$-eigenspace of $s_{\alpha}$). For $\alpha \in R$, let $\alpha^{\vee} \in \Hom_{\mathbb{R}}(V_0, \mathbb{R})$ such that
\[   s_{\alpha}(v) = v-\langle v, \alpha^{\vee} \rangle \alpha, \]
where $\langle v, \alpha^{\vee} \rangle=\alpha^{\vee}(v)$. Let $R^{\vee} \subset \Hom_{\mathbb{R}}(V_0, \mathbb{R})$ be the collection of all $\alpha^{\vee}$ ($\alpha \in R$). Let $V_0^{\vee}= \Hom_{\mathbb{R}}(V_0, \mathbb{R})$. For each $\alpha \in \Pi$, let $\omega_{\alpha} \in V$ (resp. $\omega_{\alpha}^{\vee} \in V^{\vee}$) be the fundamental weight (resp. coweight) associated to $\alpha$ i.e. $\beta^{\vee}(\omega_{\alpha})=\delta_{\alpha, \beta}$ for any $\beta\in \Pi$. 

By extending the scalars, let $V=\mathbb{C} \otimes_{\mathbb{R}} V_0$ and let $V^{\vee}=\mathbb{C} \otimes_{\mathbb{R}} V_0^{\vee}$. We call $(R, V, R^{\vee}, V^{\vee}, \Pi)$ to be a root datum. Let $n=\dim_{\mathbb{C}} V=|\Pi|$. We remark that in \cite{Ch2} we do not assume $\dim_{\mathbb{C}} V=|\Pi|$ in general.



\subsection{Graded Hecke algebras}

Let $k: \Pi \rightarrow \mathbb{C}$ be a parameter function such that $k(\alpha)=k(\alpha')$ if $\alpha$ and $\alpha'$ are in the same $W$-orbit. We shall simply write $k_{\alpha}$ for $k(\alpha)$.

\begin{definition} \label{def graded affine} \cite[Section 4]{Lu}
The graded (affine) Hecke algebra $\mathbb{H}=\mathbb{H}(\Pi,k)$ associated to a root datum $(R, V, R^{\vee}, V^{\vee}, \Pi)$ and a parameter function $k$ is an associative algebra with an unit over $\mathbb{C}$ generated by the symbols $\left\{ t_w :w \in W \right\}$ and $\left\{ f_v: v \in V \right\}$ satisfying the following relations:
\begin{enumerate}
\item[(1)] The map $w  \mapsto t_w$ from $\mathbb{C}[W]=\bigoplus_{w\in W} \mathbb{C}w  \rightarrow \mathbb{H }$ is an algebra injection,
\item[(2)] The map $v \mapsto f_v$ from $S(V) \rightarrow \mathbb{H}$ is an algebra injection,
\end{enumerate}
For simplicity, we shall simply write $v$ for $f_v$ from now on.
\begin{enumerate}
\item[(3)] the generators satisfy the following relation:
\[    t_{s_{\alpha}}v-s_{\alpha}(v)t_{s_{\alpha}}=k_{\alpha}\langle v, \alpha^{\vee} \rangle \quad \quad (\alpha \in \Pi,\ v \in V).\]
\end{enumerate}
In the remainder of this paper, we shall assume that $\mathbb{H}$ is a graded Hecke algebra obtained from an extended affine Hecke algebra in the sense of \cite[Section 9]{Lu} (also see \cite[Section 3]{BM}). The only place we need this assumption is Lemma \ref{lem bullet dual ismo} and the results depending on Lemma \ref{lem bullet dual ismo}. Other parts are also valid for more general graded Hecke algebras, including non-crystallographic types.
\end{definition}

\begin{notation} \label{not parabolic}
For any subset $J$ of $\Pi$, define $V_J$ to be the complex subspace of $V$ spanned by elements in $J$ and define $V_J^{\vee}$ be the dual space of $V_J$ lying in $V^{\vee}$. Let $R_J = V_J \cap R$ and let $R_J^{\vee}=V_J^{\vee} \cap R^{\vee}$. Let $R_J^+=R_J \cap R^+$. Let $W_J$ be the subgroup of $W$ generated by the elements $s_{\alpha}$ for $\alpha \in J$. 
For $J \subset \Pi$, let $W_J$ be the subgroup of $W$ generated by all $s_{\alpha}$ with $\alpha \in J$. Let $w_{0, J}$ be the longest element in $W_J$. Let $W^J$ be the set of minimal representatives in the cosets in $W/W_J$.

Let 
\[ V_J^{\vee, \bot} = \left\{ \gamma^{\vee} \in V^{\vee}: \gamma^{\vee}(v)=0 \mbox{ for all $v \in V_J$} \right\} \]
and 
\[ V_J^{ \bot} = \left\{ v \in V: \gamma^{\vee}(v)=0 \mbox{ for all $\gamma^{\vee} \in V_J^{\vee}$} \right\}. \]

For $J \subset \Pi$, let $\mathbb{H}_J$ be the subalgebra of $\mathbb{H}$ generated by all $v \in V$ and $t_w$ ($w \in W_J$) and let $\mathbb{H}_J^{ss}$ be the subalgebra of $\mathbb{H}$ generated by all $v \in V_J$ and $t_w$ ($w \in W_J$). We have
\[  \mathbb{H}_J \cong \mathbb{H}^{ss}_J \otimes S(V_J^{\bot})  .\]
\end{notation}

For a complex associative algebra $\mathbb{A}$ with an unit and for $\mathbb{A}$-modules $X$ and $Y$, denote by $\mathrm{Ext}^i_{\mathbb{A}}(X,Y)$ the $i$-th derived functor of $\mathrm{Hom}_{\mathbb{A}}(.,Y)$ on $X$ in the category of $\mathbb{A}$-modules. 

\begin{definition}
For a finite-dimensional $\mathbb{H}_J$-module $U$ and for $\gamma^{\vee} \in V^{\vee}$, an element $0 \neq u \in U$ is called a generalized $\gamma^{\vee}$-weight vector (or simply generalized weight vector) if $(v-\gamma^{\vee}(v))^r.u=0$ for some positive integer $r$. Such $\gamma^{\vee}$ is called a {\it weight} of $U$. 

Let $U$ be an $\mathbb{H}_J$-module. Let $\mathrm{Wgt}(U)$ be the set of weights of $U$. Let $M(U, \gamma^{\vee})$ be the space spanned by generalized weight vectors of weight $\gamma^{\vee}$. 
\end{definition}
\subsection{Central characters}
Let $J \subset \Pi$. The center $Z(\mathbb{H}_J)$ of $\mathbb{H}_J$ is naturally isomorphic to $S(V)^{W_J}$ \cite[Theorem 6.5]{Lu}, the $W_J$-invariant polynomials in $S(V)$. Each element in $Z(\mathbb{H}_J)$ acts on an simple $\mathbb{H}_J$-module $X$ by a scalar, and determines a map from $S(V)^{W_J}$ to $\mathbb{C}$. Those maps can be associated to a unique $W_J$-orbits on $V^{\vee}$. We shall call such orbit to be the {\it central character} of such module $X$.

We shall use the following result freely in the remainder of this paper. See for example \cite[Theorem I. 4.1]{BW} for similar arguments, also see \cite[Proposition 4.2.32]{Ch}.

\begin{lemma} \label{lem cc conjugate}
Let $J \subset \Pi$. Let $X$ and $Y$ be simple $\mathbb{H}_J$-modules. If $X$ and $Y$ have different central characters, then $\mathrm{Ext}^i_{\mathbb{H}_J}(X, Y) =0$ for all $i$.
\end{lemma}

\subsection{$\bullet$-dual} \label{ss duals}

We recall an anti-involution $\bullet$ studied by Opdam \cite{Op} and Barbasch-Ciubotaru \cite{BC, BC2} (with a slight variation). Define $\bullet: \mathbb{H} \rightarrow \mathbb{H}$ to be a linear anti-involution determined by
\[  v^{\bullet}=v \quad \mbox{ for $v \in V$ }, \quad t_w^{\bullet}=t_w^{-1} \quad \mbox{ for $w \in W$ }.
\]
For an $\mathbb{H}$-module $X$, the $\bullet$-operation defines a dual module denoted $X^{\bullet}$.
\begin{lemma} \label{lem bullet dual ismo}
All simple $\mathbb{H}$-modules are self $\bullet$-dual.
\end{lemma}

\begin{proof}
Let $X$ be a simple $\mathbb{H}$-module. Note that $X^{\bullet}$ has the same weight space as $X$ from the definition of $\bullet$. To prove the lemma, it suffices to show that if two simple modules have the same set of weights, then they are isomorphic, that is \cite[Theorem 5.5]{EM} and the Lusztig's reduction theorem \cite{Lu}. 
\end{proof}


\subsection{K\"unneth formula}

We have the following form of K\"unneth formula. For a proof, see for example \cite[Theorem 3.5.6]{Be} for similar arguments.

\begin{lemma} \label{lem kuneth form}
Let $J \subset \Pi$. Let $\overline{U}_1, \overline{U}_2$ be finite-dimensional $\mathbb{H}_J^{ss}$-modules and let $L_1, L_2$ be finite-dimensional $S(V_J^{\bot})$-representations. Then
\[ \mathrm{Ext}^i_{\mathbb{H}_J^{ss} \otimes S(V_J^{\bot})}(\overline{U}_1 \otimes L_1, \overline{U}_2 \otimes L_2) \cong \bigoplus_{j+k=i} \mathrm{Ext}^j_{\mathbb{H}_J^{ss}}(\overline{U}_1, \overline{U}_2) \otimes \mathrm{Ext}^k_{S(V_J^{\bot})}(L_1, L_2) .
\]
\end{lemma}

\section{Generalized standard modules} \label{s gss ext}
The main goal of this section is to see some extensions between standard modules can be constructed from extensions of corresponding tempered modules for a parabolic subalgebra. Then we shall make some constructions and further study in Section \ref{s quotient gsm}.

\subsection{ Tempered modules} \label{ss temp}

Since $V^{\vee}$ admits a natural real form, we can talk about the real part of elements in $V^{\vee}$.

\begin{definition} \label{def temp ds}
Let $J \subset \Pi$. A (not necessarily irreducible) $\mathbb{H}_J$-module $U$ is said to be an $\mathbb{H}_J$-{\it tempered module} if the real part of any weight $\gamma^{\vee} \in V^{\vee}$ of $X$ has the form:
\begin{align} \label{eqn temp}
   \mathrm{Re} \gamma^{\vee} = \sum_{\alpha \in J} a_{\alpha} \alpha^{\vee} +\sum_{\alpha \in \Pi \setminus J}b_{\alpha} \omega_{\alpha}^{\vee}, \quad \mbox{where } a_{\alpha} \leq 0,\ b_{\alpha} >0 .
\end{align}
If $U$ is an irreducible $\mathbb{H}_J$-tempered module, denote 
\[  \nu(J, U) = \sum_{\alpha \in \Pi \setminus J} b_{\alpha} \omega^{\vee}_{\alpha} ,
\]
where $ \sum_{\alpha \in \Pi \setminus J} b_{\alpha} \omega^{\vee}_{\alpha}$ is the second term of the left hand side of (\ref{eqn temp}). Note that our assumption of the irreducibility on $U$ assures that the term $\sum_{\alpha \in \Pi \setminus J} b_{\alpha} \omega^{\vee}_{\alpha} $ is independent of the choice of a weight for $U$. 

An $\mathbb{H}$-module $X$ is said to be a {\it discrete series} if $X$ is an $\mathbb{H}$-tempered module with all the inequalities for $a_{\alpha}$ in (\ref{eqn temp}) being strict (i.e. $a_{\alpha} <0$ for all $\alpha \in \Pi$). 
\end{definition}

Here our terminology of $\mathbb{H}_J$-tempered modules follows \cite{KR}. It sometimes does not coincide with other definitions in the literature (e.g. \cite[Definition 1.4]{Ev}). Our terminology is more convenient for discussions, although it is not quite a direct analog of tempered representations of $p$-adic groups.

\subsection{Generalized standard modules and the Langlands classification}

For any $J \subset \Pi$ and any $\mathbb{H}_J$-module $U$, we sometimes write $I(J,U)$ for $\mathbb{H} \otimes_{\mathbb{H}_J}U$. For any $\mathbb{H}$-module $X$, write $\mathrm{Res}_{\mathbb{H}_J} X$ to be the restriction of $X$ to an $\mathbb{H}_J$-module. (The notation $\mathrm{Res}$ will also be similarly used for other algebras such as $\mathbb{H}^{ss}_J$.)

\begin{definition} \label{def generalized St}
Let $\Xi_L^g$ be the collection of all pairs $(J,U)$ with $J \subset \Pi$ and $U$ being a finite-dimensional indecomposable $\mathbb{H}_J$-tempered module. Let $\Xi_L$ be the subset of $\Xi_L^g$ containing all pairs $(J, U)$ with $U$ being an {\it irreducible} $\mathbb{H}_J$-tempered module. For $(J, U) \in \Xi_L$, $I(J,U)$ is usually called a {\it standard module} in the literature. For $(J,U) \in \Xi^g_L$, we shall call the module $I(J,U)$ to be a {\it generalized standard module}. 


For $(J,U) \in \Xi_L$, let $L(J,U)$ be the unique simple quotient of $I(J,U)$ and let $N(J,U)$ be the unique maximal proper submodule of $I(J,U)$. For any irreducible $\mathbb{H}$-module $X$, $X$ is isomorphic to $L(J,U)$ for some $(J,U) \in \Xi_L$. (See Langlands classification \cite{Ev, KR} for the details.) For such $X$, denote $\nu(X)=\nu(J,U)$. 

For $(J,U) \in \Xi_L^g$, we can also define $\nu(J,U)=\nu(J,U')$, where $U'$ is a composition factor of $U$. It is well-defined by Lemma \ref{lem com factor same nu} below. We shall extend the notion of $L(J,U)$ to all $(L,U) \in \Xi_L^g$ in Definition \ref{def quotient gen} later. 

\end{definition}

\begin{lemma} \label{lem com factor same nu}
Let $(J,U) \in \Xi_L^g$. Let $U_1, U_2$ be (simple) composition factors of $U$. Then $\nu(J, U_1)=\nu(J,U_2)$. 
\end{lemma}

\begin{proof}
Recall that $\mathbb{H}_J \cong \mathbb{H}_J^{ss} \otimes S(V_J^{\bot})$. Note that any irreducible $\mathbb{H}_J$-module is isomorphic to $\overline{U} \otimes L$ for some irreducible $\mathbb{H}_J$-module $\overline{U}$ and $1$-dimensional $S(V_J^{\bot})$-module $L$. For two simple $\mathbb{H}_J$ modules $\overline{U}_1 \otimes L_1$ and $\overline{U}_2 \otimes L_2$, if $L_1 \not\cong L_2$ (equivalently $\nu(J, \overline{U}_1 \otimes L_1) \neq \nu(J, \overline{U}_2 \otimes L_2)$), then 
\begin{align} \label{eqn diff nu ext 0}
 \mathrm{Ext}^1_{\mathbb{H}_J}(\overline{U}_1 \otimes L_1, \overline{U}_2 \otimes L_2) =0 
\end{align}
by Lemma \ref{lem kuneth form} and $\mathrm{Ext}^1_{S(V_J^{\bot})}(L_1, L_2)=0$. 

Now we consider $(J,U) \in \Xi_L^g$ and prove by an induction on the length of $U$. If the length of $U$ is $1$, then the statement is clear. Let $U'$ be a maximal submodule of $U$ such that any composition factors $U_1, U_2$ of $U'$ satisfy $\nu(J,U_1)=\nu(J,U_2)$. Note that $U'\neq 0$ since simple submodules of $U$ always satisfy that property. Now we have a short exact sequence of the following form:
\[ 0 \rightarrow U' \rightarrow U \rightarrow U/U' \rightarrow 0 .\]
Since $U/U'$ is of finite length, we can write $U/U'$ as $U/U' = N_1 \oplus \ldots \oplus N_i$, where all $N_j$ are indecomposable modules. By the inductive hypothesis, all $N_j$ have the property that all the composition factors $N'_a, N'_b$ of $N_i$ satisfy $\nu(J, N'_a)=\nu(J,N'_b)$. Now by using the maximality of our choice of $U'$ and using (\ref{eqn diff nu ext 0}), we have 
\[  \mathrm{Ext}^1_{S(V_J^{\bot})}(U', N_j)=0  \]
for all $j$. Now we have that $U \cong U' \oplus N_1\oplus \ldots \oplus N_i$. By the indecomposability of $U$, we have $U=U'$ as desired.
\end{proof}

We recall the dominance ordering on $V_0^{\vee}$. For any two elements $\gamma^{\vee}_1, \gamma^{\vee}_2 \in V^{\vee}_0$, write $\gamma^{\vee}_1 \leq \gamma^{\vee}_2$ if 
\[   \gamma_2^{\vee}-\gamma_1^{\vee} =\sum_{\alpha \in \Pi} a_{\alpha} \alpha^{\vee} \]
for some $a_{\alpha} \geq 0$. We write $\gamma^{\vee}_1 < \gamma_2^{\vee}$ if $\gamma^{\vee}_1 \leq \gamma^{\vee}_2$ and $\gamma^{\vee}_1 \neq \gamma^{\vee}_2$. 

We recall some results related to the Langlands classification. We need more notations. For any $\gamma^{\vee} \in V^{\vee}$, the real part of $\gamma^{\vee}$ can be uniquely written as 
\begin{align} \label{eqn unique express} \mathrm{Re} \gamma^{\vee} = \sum_{\alpha \in J} a_{\alpha} \alpha^{\vee} +\sum_{\alpha \in \Pi \setminus J} b_{\alpha} \omega^{\vee}_{\alpha},
\end{align}
for some $J\subset \Pi$, $a_{\alpha}<0$ and $b_{\alpha} \geq 0$ (\cite[Ch VIII Lemma 8.59]{Kn0}, c.f. $\nu(J,U)$ in Definition \ref{def temp ds}). For such $\gamma^{\vee}$, define 
\[  \nu(\gamma^{\vee}) =\sum_{\alpha \in \Pi \setminus J} b_{\alpha} \omega^{\vee}_{\alpha} .
\]

We shall freely use the fact that $\langle \omega_{\alpha}, \omega_{\beta}^{\vee} \rangle \geq 0$ for $\alpha, \beta \in \Pi$ \cite[Ch VIII Lemma 8.57]{Kn0} below.
\begin{lemma} \label{lem no WJ conj}
Let $(J,U) \in \Xi_L^g$ and let $\gamma^{\vee} \in V^{\vee}$. If $\nu(\gamma^{\vee}) < \nu(J,U)$, then $\gamma^{\vee}$ is not $W_J$-conjugate to any weight of $U$. 
\end{lemma}

\begin{proof}
By definitions, there exists $J' \subset \Pi$ such that 
\[  \mathrm{Re} \gamma^{\vee} = \sum_{\alpha \in J' } a_{\alpha} \alpha^{\vee} + \nu(\gamma^{\vee}) 
\]
for $a_{\alpha} \leq 0$ and $b_{\alpha}>0$. By simple linear algebra we can also write
\[ \mathrm{Re} \gamma^{\vee} =\sum_{\alpha \in J} a_{\alpha}' \alpha^{\vee} + \sum_{\alpha \in \Pi \setminus J} b_{\alpha}'\omega_{\alpha}^{\vee} ,\]
where $a_{\alpha}' ,b_{\alpha}' \in \mathbb{R}$. Set $\nu'=\sum_{\alpha \in \Pi \setminus J} b_{\alpha}' \omega_{\alpha}^{\vee}$. To prove the lemma, it suffices to show that $\nu'$ is not equal to $\nu(J,U)$. 
For $\alpha_0 \in J$, 
\[ a_{\alpha_0}'+\langle \omega_{\alpha_0}, \nu' \rangle= \langle \omega_{\alpha_0}, \mathrm{Re} \gamma^{\vee} \rangle \leq\langle \omega_{\alpha_0}, \nu(\gamma^{\vee})\rangle  \leq \langle \omega_{\alpha_0}, \nu(J,U) \rangle .
\]
Hence if $\nu'=\nu(J,U)$, then all $a_{\alpha_0'}\leq 0$. By the uniqueness property in the expression (\ref{eqn unique express}), we have $J=J'$ and so $\nu(\gamma^{\vee})=\nu'=\nu(J,U)$. This gives a contradiction. Hence $\nu' \neq \nu(J,U)$ as desired.
\end{proof}


\begin{lemma} \label{lem decompose for lang class}
Let $(J, U) \in \Xi_L^g$. Then
\[  \mathrm{Res}_{\mathbb{H}_J} I(J, U) \cong U \oplus Y,
\]
where $Y$ is an $\mathbb{H}_J$-module such that for any weight $\gamma^{\vee}$ of $Y$, $\nu(\gamma^{\vee}) < \nu(J, U)$. Moreover, 
\begin{enumerate}
\item any composition factors of $Y$ have $\mathbb{H}_J$-central characters different from any composition factors of $U$,
\item Suppose $(J,U) \in \Xi_L$. Then for any composition factors $Z$ of $I(J,U)$, if $Z$ is not isomorphic to $L(J,U)$, then $\nu(Z) < \nu(J,U)$. 
\end{enumerate}
\end{lemma}

\begin{proof}
This is the consequence of the proof for the Langlands classification \cite{Ev}. We provide some explanations. 

By Frobenius reciprocity, we have the following short exact sequence for $\mathbb{H}_J$-modules:
\[ 0 \rightarrow U \rightarrow \mathrm{Res}_{\mathbb{H}_J} I(J, U) \rightarrow Y  \rightarrow 0 ,\]
where $Y = \mathrm{Res}_{\mathbb{H}_J} I(J, U)/U$. Then using the arguments as in \cite{Ev} which use Geometric Lemmas of Langlands (also see proof of \cite[Theorem 2.4]{KR} and proof of \cite[Proposition 7.1.81]{Ch}), we have that for all the weights $\gamma^{\vee}$ of $Y$, $\nu(\gamma^{\vee}) <\nu(J,U)$. (We remark that while $U$ is taken to be irreducible in \cite{Ev} and \cite[Theorem 2.4]{KR}, their argument still works by our Lemma \ref{lem com factor same nu}.) This implies that any composition factors of $Y$ have different $\mathbb{H}_J$-central characters from any composition factors of $U$ since any weight of $Y$ is not $W_J$-conjugate to that of $N$ (Lemma \ref{lem no WJ conj}). Hence 
\[ \mathrm{Ext}^1_{\mathbb{H}_J}(Y, U)=0 \]
and so $\mathrm{Res}_{\mathbb{H}_J} I(J, U) \cong U \oplus N$ as desired. Other assertions and details can be referred to \cite{Ev, KR}.
\end{proof}

\subsection{Extensions}

The following vanishing result is useful in computing some extensions.

\begin{proposition} \label{prop ext lower langlands para}
Let $(J_1,U_1), (J_2, U_2) \in \Xi_L$. If either $\nu(J_1, U_1)$ and $\nu(J_2, U_2)$ are incomparable or $\nu(J_2, U_2) < \nu(J_1, U_1)$, then 
\[ \mathrm{Ext}^i_{\mathbb{H}}(I(J_1, U_1), L(J_2, U_2)) =\mathrm{Ext}^i_{\mathbb{H}}(I(J_1, U_1), I(J_2, U_2))=0 .\]
\end{proposition}

\begin{proof}
Suppose $\mathrm{Ext}^i_{\mathbb{H}}(I(J_1, U_1), L(J_2, U_2)) \neq 0$, equivalently $\mathrm{Ext}^i_{\mathbb{H}_{J_1}}(U_1, \mathrm{Res}_{\mathbb{H}_{J_1}}L(J_2, U_2)) \neq 0$ by Shapiro's Lemma. Then by considering the $\mathbb{H}_{J_1}$-central character and using Lemma \ref{lem cc conjugate}, there exists an $\mathbb{H}_{J_1}$-composition factor $U'$ of $\mathrm{Res}_{\mathbb{H}_{J_1}}L(J_2, U_2)$ such that $U'$ has a weight whose real part has the form
\[  \sum_{\alpha \in J_1} a_{\alpha}' \alpha^{\vee} +\nu(J_1, U_1).
\]
Using the Langlands classification on the irreducible $\mathbb{H}^{ss}_{J_1}$-module $\mathrm{Res}_{\mathbb{H}_{J_1}^{ss}} U'$, we can further get a weight, denoted $\gamma^{\vee}$, of $U'$ of the following form:
\[\mathrm{Re}(\gamma^{\vee})=\left(\sum_{\alpha \in J'} a_{\alpha} \alpha^{\vee}+ \sum_{\alpha \in J_1 \setminus J'} b_{\alpha}\overline{\omega}_{\alpha}^{\vee} \right) +\nu(J_1, U_1),
\]
with $J' \subset J_1$, $a_{\alpha} \leq 0$ and $b_{\alpha} >0$, where $\overline{\omega}_{\alpha}^{\vee}$ are fundamental coweights in $V_{J_1}^{\vee}$ corresponding to $\alpha$ (i.e. $\overline{\omega}_{\alpha}^{\vee} \in V_{J_1}$, $\overline{\omega}_{\alpha}^{\vee}(\beta)=\delta_{\alpha, \beta}$ for $\alpha, \beta \in R_{J_1}$). By \cite[Ch VIII Lemma 8.57]{Kn0}, $\overline{\omega}_{\alpha}^{\vee}$ is non-negative sum of simple coroots. Hence 
\[   \sum_{\alpha \in J'} a_{\alpha} \alpha^{\vee}+ \nu(J_1, U_1) \leq \mathrm{Re}(\gamma^{\vee} ).
\]
Then by geometric lemmas of Langlands (\cite{La}, see \cite[Ch VIII Lemmas 8.56 and 8.59]{Kn0}),
\[   \nu(J_1, U_1)= \nu\left (\sum_{\alpha \in J'} a_{\alpha} \alpha^{\vee}+\nu(J_1, U_1)\right) \leq \nu(\gamma^{\vee}).
\]
Now by Lemma \ref{lem decompose for lang class}, we have 
\[ \nu(J_1, U_1) \leq \nu(\gamma^{\vee}) \leq \nu(J_2, U_2)  .
\]
This proves $\mathrm{Ext}^i_{\mathbb{H}}(I(J_1, U_1), L(J_2, U_2)) =0$ if $\nu(J_1, U_1) \not\leq \nu(J_2, U_2)$.

We now consider $\mathrm{Ext}^i_{\mathbb{H}}(I(J_1, U_1), I(J_2, U_2))$  and suppose $\nu(J_1, U_1) \not\leq \nu(J_2, U_2)$. Then for any composition factor $X$ of $I(J_2, U_2)$, $\nu(J_1, U_1) \not\leq \nu(X)$ by Lemma \ref{lem decompose for lang class}. Hence from what we have just proved, we have $\mathrm{Ext}^i_{\mathbb{H}}(I(J_1, U_1), X)=0$. This implies that $\mathrm{Ext}^i_{\mathbb{H}}(I(J_1, U_1), I(J_2,U_2))=0$.
\end{proof}

\begin{remark}
One may compare with a vanishing result of Kato \cite{Ka} (see the paragraph below \cite[Theorem C]{Ka}, also see \cite[Section 8]{Lu0}).
\end{remark}
\subsection{Indecomposability}

We show below that a generalized standard module is indecomposable.

\begin{proposition} \label{prop indecomp standard}
Let $(J,U) \in \Xi_L^g$. Then $I(J,U)$ is also an indecomposable $\mathbb{H}$-module. 
\end{proposition}

\begin{proof}
Let $(J,U) \in \Xi_L^g$ and let $X=I(J, U)$. Let 
\[  0 \rightarrow Y \stackrel{f}{\rightarrow} X \stackrel{g}{\rightarrow} Z \rightarrow 0 \]
be a short exact sequence. Suppose the short exact sequence splits and we shall show that either $Y$ or $Z$ is zero. Since the short exact sequence splits, there exists an $\mathbb{H}$-map $t: Z \rightarrow X$ such that $g \circ t=\mathrm{Id}$. Since $\mathrm{Res}_{\mathbb{H}_J}$ is an exact functor, we have a short exact sequence
\[ 0 \rightarrow \mathrm{Res}_{\mathbb{H}_J} Y \rightarrow \mathrm{Res}_{\mathbb{H}_J}X \rightarrow \mathrm{Res}_{\mathbb{H}_J}Z \rightarrow 0 .
\]
We also have an exact sequence of the form 
\[   0 \rightarrow f^{-1}(U) \stackrel{f|_{f^{-1}(U)}}{\rightarrow} U \stackrel{g|_U}{\rightarrow} g(U) \rightarrow 0,
\]
where we identify the $\mathbb{H}_J$-subspace $1 \otimes U$ of $I(J,U)$ with $U$. Since $t(g(U))$ has an $\mathbb{H}_J$-central character as $U$, $t(g(U))\subset U$ by Lemma \ref{lem decompose for lang class}. Thus it makes sense to write $g|_{U} \circ t|_{g(U)} =\Id$. Since $U$ is an indecomposable $\mathbb{H}_J$-map, we have either $f^{-1}(U)=0$ or $g(U)=0$. Suppose $g(U)=0$. Then by Frobenius reciprocity, $g=0$ and hence $Z=0$. Suppose $f^{-1}(U)=0$. Then $t|_{g(U)}$ is surjective onto $U$. This implies that $1 \otimes U \subset t(g(U)) \subset t(Z)$. By applying the $\mathbb{H}$-action on both sides of the inclusion, we have $X \subset t(Z)$ and so $t(Z)=X$ (i.e. $t$ is surjective). Hence $t$ is an isomorphism and so is $g$. Thus $Y=0$ as desired. 
\end{proof}

\begin{remark}
Proposition \ref{prop indecomp standard} is not true in general if $U$ is not $\mathbb{H}_J$-tempered. For example, let $R$ be the root system of type $A_1$. Assume $k_{\alpha} \neq 0$ for all $\alpha \in \Pi$. Let $U'$ be the unique $2$-dimensional indecomposable $S(V)$-module with the weight equal to $0$. In this case, $U'$ is not $\mathbb{H}_{\emptyset}$-tempered and one can verify that $I(\emptyset, U')$ is semisimple and of length $2$. (In our terminology, $I(\emptyset, U')$ is $\mathbb{H}$-tempered.)
\end{remark}

The following result and its proof mainly guides for constructions in Section \ref{s quotient gsm}.

\begin{lemma} \label{lem equiv cat 1}
Let $X$ be a finite-dimensional $\mathbb{H}$-module. Suppose there exists a fixed $J \subset \Pi$ and a fixed $\nu^{\vee} \in V_J^{\vee, \bot} \cap V_0^{\vee}$ such that $X$ admits a filtration $0 = F_0 \subset \ldots \subset F_{r-1} \subset F_r =X$ with the subquotients $F_{i+1}/F_i \cong I(J,U_i)$ for some $(J, U_i) \in \Xi_L^g$ satisfying $\nu(J, U_i)=\nu$. Then there exists an $\mathbb{H}_J$-tempered module $U$ such that $X$ is isomorphic to $I(J,U)$. Moreover, the $\mathbb{H}_J$-tempered module $U$ is unique, up to isomorphism.
\end{lemma}

\begin{proof}
By using an inductive argument, it suffices to prove the case when $X$ admits such filtration of length 2 (i.e. $r=2$).
For $i=1,2$, let $P_{U_i}^{\bullet}$ be a projective resolution of $U_i$. We write as
\[  P_{U_i}^2 \stackrel{d_2}{\rightarrow} P_{U_i}^1 \stackrel{d_1}{\rightarrow} P_{U_i}^0 \rightarrow U_i \rightarrow 0 .
\]
Then $I(J, P_{U_i}^k)$ is still a projective  resolution for $I(J,U_i)$ with the differential maps denoted by $d_k'$.
The map is determined by
\begin{align}\label{diag induction} \xymatrix{ 
\cdots    \ar[r] & I(J, P^1_{U_1}) \ar[d]^{f} \ar[r]^{d_2'} & I(J, P^0_{U_1}) \ar[d]\ar[r] & I(J, U_1) \ar[r] \ar[d]^{id}  &0  \\
0   \ar[r] & I(J, U_2) \ar[r] & X  \ar[r] &  I(J,U_1)   \ar[r]  & 0, }
\end{align}
where $M$ is the pushout of $d_2': I(J,P_{U_1}^1) \rightarrow I(J, P_{U_1}^0)$ and $f': I(J,P_{U_1}^1) \rightarrow I(J,U_2)$. Explicitly, 
\[  X \cong (I(J,U_2) \oplus I(J,P^0_{U_1} ))/\left\{ (f(p), -d'_2(p)) : p \in I(J, P^1_{U_1}) \right\}.
\]

Recall that $I(J, P_{U_i}^k)=\mathbb{H} \otimes_{\mathbb{H}_J} P^k_{U_i}$.  Now we consider the identification
\[\mathrm{Hom}_{\mathbb{H}}(I(J, P_{U_1}^1), I(J, P_{U_2})) \cong \mathrm{Hom}_{\mathbb{H}_J}(P_{U_1}^1, \mathrm{Res}_{\mathbb{H}_J}I(J, P_{U_2})) \cong  \mathrm{Hom}_{\mathbb{H}_J}(P_{U_1}^1, U_2) \oplus \mathrm{Hom}_{\mathbb{H}_J}(P_{U_1}^1, N),
\]
where $N$ is an $\mathbb{H}_J$-module whose composition factors have weights  $\nu_1^{\vee}$ satisfying $\nu_1^{\vee} < \nu(J, U_2)=\nu^{\vee}$ as in Lemma \ref{lem decompose for lang class}. Under such identification, we write $f$ as $(f_1, f_2)$. Now since $\mathrm{Ext}^1_{\mathbb{H}_J}(U_1, N)=0$ by using Proposition \ref{prop ext lower langlands para}, $(f_1, 0)$ and $(f_1, f_2)$ determine the same cohomology class of $\mathrm{Ext}^1_{\mathbb{H}}(I(J, U_1), I(J,U_2)) \cong \mathrm{Ext}^1_{\mathbb{H}_J}(U_1, U_2) \oplus \mathrm{Ext}^1_{\mathbb{H}_J}(U_1 N)$. Let $f' =(f_1,0)$. By the identification, we now have $f'(1 \otimes P^1_{U_1}) \subset 1 \otimes U_2$. By definition, we have $d_2'(1 \otimes P^1_{U_1}) \subset 1\otimes P_{U_1}^0$. Hence, we now have maps $\widetilde{f}': P^1_{U_1} \rightarrow U_2$ and $\widetilde{d}_2': P^1_{U_1} \rightarrow P_{U_1}^0$ naturally arisen from $f'$ and $d_2'$ respectively. Let
\[M =(I(J,U_2) \oplus I(J,P^0_{U_1} ))/\left\{ (f'(p), d'_2(p)) : p \in I(J, P^1_{U_1}) \right\}.
\]
By the construction of the Yoneda extension, $X \cong M$. We now also define 
\[  M_U = (U_2 \oplus P^0_{U_1})/\left\{ (\widetilde{f}'(1 \otimes p), \widetilde{d}_2'(1 \otimes p)) : p \in P^1_{U_1} \right\} .
\]
It is straightforward to verify $ M \cong I(J, M_U)$. Hence $X \cong I(J, M_U)$ as desired.

It remains to prove the uniqueness. Suppose there exists $\mathbb{H}_J$-tempered modules $U_1$ and $U_2$ such that $X \cong I(J,U_1) \cong I(J,U_2)$. Let $f: I(J,U_1) \rightarrow I(J,U_2)$ be an isomorphism. By using Proposition \ref{prop indecomp standard}, we may first reduce $U_1$ and $U_2$ to be indecomposable and then show that $f(1 \otimes U_1) = 1 \otimes U_2$ by Lemma \ref{lem decompose for lang class}. Then $f$ defines an $\mathbb{H}_J$-isomorphism between $U_1$ and $U_2$. 
\end{proof}



\subsection{$S(V)$-type extensions} \label{ss sv type ext}

We shall refine our study to a certain class of extensions  which arises from extensions of representations for polynomial rings. Such extensions shall be called $S(V)$-type extensions and have close connections to Jantzen filtrations discussed later. Constructions and more study on $S(V)$-type extensions will be carried out in Section \ref{s quotient gsm} after the necessary tools have been explained.

\begin{lemma} \label{lem nat isom ext}
Let $(J, U_1), (J, U_2) \in \Xi_L$ with $\nu(J, U_1)=\nu(J, U_2)$. Via the natural identification of $\mathbb{H}_J \cong \mathbb{H}_J^{ss} \otimes S(V_J^{\bot})$, we have
\[   U_1 \cong \overline{U}_1 \otimes L_1 ,\   U_2 \cong \overline{U}_2 \otimes L_2  \]
such that $\overline{U}_i$ is an $\mathbb{H}_J^{ss}$-module and $L_i$ is a 1-dimensional $S(V_J^{\bot})$-representation. Then
\begin{align} \label{eqn ext standard same}
  \mathrm{Ext}^1_{\mathbb{H}}(I(J,U_1), I(J,U_2)) &\cong \mathrm{Hom}_{\mathbb{H}_J^{ss}}(\overline{U}_1, \overline{U}_2) \otimes \mathrm{Ext}^1_{S(V_J^{\bot})}(L_1, L_2) \\
	                                                & \quad \quad \oplus \mathrm{Hom}_{S(V_J^{\bot})}(L_1, L_2) \otimes \mathrm{Ext}^1_{\mathbb{H}_J^{ss}}(\overline{U}_1, \overline{U}_2)  .
\end{align}

\end{lemma}

\begin{proof}
By Shapiro's Lemma, Lemma \ref{lem no WJ conj} and Lemma \ref{lem decompose for lang class}, 
\[ \mathrm{Ext}^i_{\mathbb{H}}(I(J, U_1), I(J, U_2)) \cong \mathrm{Ext}^i_{\mathbb{H}_J}(U_1, U_2) \oplus \mathrm{Ext}^i_{\mathbb{H}_J}(U_1, Y) .\]
By Lemma \ref{lem decompose for lang class}, the $\mathbb{H}_J$-central characters of $U_1$ and $Y$ are different and hence $\mathrm{Ext}^i_{\mathbb{H}}(U_1, Y)=0$. The $\mathrm{Ext}^i_{\mathbb{H}_J}(U_1, U_2)$ is naturally isomorphic to the term in the left hand side of (\ref{eqn ext standard same}) by Lemma \ref{lem kuneth form}.
\end{proof}


\begin{definition} \label{def sv extension}
We use the notation in Lemma \ref{lem nat isom ext}. By composing the isomorphism in Lemma \ref{lem nat isom ext} and a projection map to one of factors, we obtain a map 
\[ \mathrm{pr}_{\mathbb{H}^{ss}_J}:  \mathrm{Ext}^1_{\mathbb{H}}(I(J, U_1), I(J, U_1)) \rightarrow  \mathrm{Hom}_{S(V_J^{\bot})}(L_1, L_1) \otimes \mathrm{Ext}^1_{\mathbb{H}^{ss}}(\overline{U}_1, \overline{U}_1) \cong \mathrm{Ext}^1_{\mathbb{H}^{ss}}(\overline{U}_1, \overline{U}_1)  , \]
\[ \mathrm{pr}_{S(V_J^{\bot})}:  \mathrm{Ext}^1_{\mathbb{H}}(I(J, U_1), I(J, U_1)) \rightarrow \mathrm{Hom}_{\mathbb{H}_J^{ss}}(\overline{U}_1, \overline{U}_1) \otimes \mathrm{Ext}^1_{S(V_J^{\bot})}(L_1, L_1) \cong \mathrm{Ext}^1_{S(V_J^{\bot})}(L_1, L_1) . \]

Let $\zeta \in \mathrm{Ext}^1_{\mathbb{H}}(I(J, U_1), I(J, U_2))$. We say the element $\zeta$ is a {\it $S(V)$-extension} if $\mathrm{pr}_{\mathbb{H}^{ss}_J}(\zeta)=0$ and $\mathrm{pr}_{S(V_J^{\bot})} \neq 0$. 


Let $(J,U) \in \Xi_L$. Let $\eta^{\vee} \in \mathrm{Ext}^1_{\mathbb{H}}(I(J,U), I(J,U))$ be a $S(V)$-extension. We call $X$ is {\it $(J,U,\eta^{\vee})$-$S(V)$-type} (or simply strict $S(V)$-type) if there exists a filtration on $X$ of the form: 
\[0 \subset X_1 \subset X_2 \subset \ldots \subset X_l =X\]
 such that for all $i$, $X_i/X_{i-1} \cong I(J,U)$ and the short exact sequence
\begin{align} \label{eqn short exact sv type}
 0 \rightarrow X_{i}/X_{i-1} \rightarrow X_{i+1}/X_{i-1} \rightarrow X_{i+1}/X_i \rightarrow 0 
\end{align}
corresponds to $\eta^{\vee}$ under the Yoneda correspondence.

It is indeed not necessary to fix one $\eta^{\vee}$ for all short exact sequences in the above definition and allow a larger class of modules. Our approach later can also deal with some of those modules, but it requires more set-up. For our purpose of studying first extension, those strict $S(V)$-types will suffice.


\end{definition}

\section{Intertwining operators } \label{s intertwin operat}

Intertwining operators for parabolically induced modules are the major tools for our computation in this paper. Some treatments in this section and the Appendix B are similar to \cite{KR} and \cite{Re0}. The intertwining operator defines the Jantzen filtration of a generalized standard modules which will be discussed in Section \ref{s jan filt}.

The main result in this section is Proposition \ref{prop intertwin properties}, which gives a description of an intertwining operator for a generalized standard module. The image of the intertwining operator also defines a quotient $L(J,U)$ for $(J,U) \in \Xi_L^g$ in Definition \ref{def quotient gen}.

\subsection{Intertwining elements}

In this section, we fix $J \subseteq \Pi$ and fix a finite-dimensional $\mathbb{H}_J$-module $U$. We shall define some intertwining element, which involves inverting some elements in $S(V)$. To deal with such matter in a proper way, we define some notations below.

\begin{definition}

Suppose $J \neq \Pi$. Let $A$ be the multiplicative closed set in $S(V)$ which contains $1$ and all the elements of the form $(v_1- c_1)\ldots (v_k-c_k)$, where $v_i \in V \setminus V_J$ and $c_k \in \mathbb{C}$. Let $\mathcal O(J)=A^{-1}S(V)$ be the localization of the ring $S(V)$ by $A$. If $J=\Pi$, simply set $\mathcal O(J)=S(V)$. 

\end{definition}

\begin{lemma} \label{lem rat sturcture}
$\mathbb{H}_J \otimes_{S(V)} \mathcal O(J)$ has a natural algebra structure such that $\mathbb{H}_J$ embeds naturally into $\mathbb{H}_J \otimes_{S(V)} \mathcal O(J)$ as a subalgebra.

\end{lemma}

\begin{proof}
For $v \in V \setminus V_J$, the relation between $t_{s_{\alpha}} \in \mathbb{H}_J$ ($\alpha \in J$) and $\frac{1}{v} \in \mathcal O(J)$ is given by 
\[(t_{s_{\alpha}} \otimes 1) (1 \otimes \frac{1}{v-c}) - (1 \otimes \frac{1}{s_{\alpha}(v)-c})(t_{s_{\alpha}} \otimes 1)=\frac{k_{\alpha}\alpha^{\vee}(v)}{(v-c)(s_{\alpha}(v)-c)} .\]
It is straightforward to check that $s_{\alpha}(v) \notin V_J$ and so the relation is well-defined. The map $ h \mapsto h\otimes 1$ from $\mathbb{H}_J$ to $\mathbb{H}_J \otimes_{S(V)} \mathcal O(J)$ defines the natural embedding.
\end{proof}

Define $\mathcal H_J=\mathbb{H}_J \otimes_{S(V)} \mathcal O(J)$ (which is an algebra by Lemma \ref{lem rat sturcture}). Define $\mathcal H^J=  \mathbb{H} \otimes_{\mathbb{H}_J} (\mathbb{H}_J \otimes_{S(V)} \mathcal O(J))=\mathbb{H} \otimes_{\mathbb{H}_J} \mathcal H_J$ (which does not have a natural algebraic structure) and we shall regard $\mathcal{H}_J$ as an $(\mathbb{H}, \mathcal{H}_J)$-bimodule (by the left and right multiplications respectively). For $w \in W$, we shall simply write $t_w$ for $t_w \otimes (1 \otimes 1)$ as an element in $\mathcal H_J$. For $q \in \mathcal O(J)$, we shall simply write $q$ for $(1 \otimes 1 \otimes q)$ as an element in $\mathcal H^J$. We also have other similar notations such as $t_wq$ for $t_w \otimes 1 \otimes q$.




\begin{definition} \label{def curl HJ mod}
Let $p=\dim V_J$ and let $\left\{ \omega_1^{\vee},\ldots, \omega_{n-p}^{\vee} \right\}$ be a basis for $V_J^{\vee,\bot}$. Denote $\mathbb{C}(\mathbf a_1, \ldots, \mathbf a_{n-p})$ be the algebra of rational functions with indeterminantes $\mathbf a_1, \ldots, \mathbf a_{n-p}$ over $\mathbb{C}$. 

Let $U$ be an $\mathbb{H}_J$-module. Let $\eta_{\mathbf{a}}^{\vee}=\mathbf a_1\omega_1^{\vee}+\ldots +\mathbf a_{n-p}\omega_{n-p}^{\vee}$, which will be regarded as a (natural) function from $V$ to $\mathbb{C}(\mathbf a_1, \ldots, \mathbf a_{n-p})$. Define $U_{\mathbf{a}}$ to be an $\mathcal{H}_J$-module such that 
$U_{\mathbf{a}}$ is isomorphic to $\mathbb{C}(\mathbf a_1,\ldots, \mathbf a_{n-p}) \otimes_{\mathbb{C}} U$ as vector spaces and the action of $\mathcal H_J$ is determined by
\[    \pi_{U_{\mathbf{a}}}(t_w)(b\otimes u)= b \otimes \pi_{U}(t_w) u \quad \mbox{ for $w \in W$ }
\]
\[   \pi_{U_{\mathbf{a}}}(v)(b\otimes u)= b \otimes \pi_{U}(v) u + \eta_{\mathbf{a}}^{\vee}(v) b \otimes u \quad \mbox{ for $v\in V$ }
\]
For $v \in V \setminus V_J$, $\pi_{U_{\mathbf{a}}}(v)$ is invertible for generic values of $(\mathbf a_1, \ldots, \mathbf a_{n-p})$. Hence the $\mathcal H_J$-action on $U_{\mathbf{a}}$ is well-defined. For an element $b \otimes u \in U_{\mathbf a}$, we shall simply write $bu$ for $b \otimes u$. There is a natural multiplication of $\mathbb{C}(\mathbf a_1, \ldots, \mathbf a_{n-p})$ on $U_{\mathbf a}$ and we shall consider $U_{\mathbf a}$ to be an $\mathcal H_J$-module over $\mathbb{C}(\mathbf a_1,\ldots, \mathbf a_{n-p})$. 

Fix a $\mathbb{C}$-basis $\left\{ u_1, \ldots, u_k \right\}$ for $U$.  We consider the tensor product $\mathcal H^J \otimes_{\mathcal H_J} U_{\mathbf{a}}$, which will be regarded as an $\mathbb{H}$-module via the left multiplication of $\mathbb{H}$ on $\mathcal H^J$. For any element $x_{\mathbf a} \in \mathcal H^J \otimes_{\mathcal H_J} U_{\mathbf{a}}$, $x_{\mathbf a}$ can be written into the form
\begin{align} \label{eqn standard form}
x_{\mathbf a}= \sum_{w \in W^J} t_w \otimes \left( \sum_{i=1}^k b_{w,i} u_i \right),
\end{align}
where $b_{w,i} \in \mathbb{C}(\mathbf{a}_1,\ldots, \mathbf{a}_{n-p})$. We say that $x_{\mathbf a}$ is {\it holomorphic at $0$} if each $b_{w,i}$ is holomorphic at $(0,\ldots, 0)$. It is easy to see that the definition of holomorphicity is independent of the choice of a basis for $U$. Then for an holomorphic element $x_{\mathbf a} \in \mathcal H^J \otimes_{\mathcal H_J} U_{\mathbf{a}}$ with the form (\ref{eqn standard form}), define the specialization $|_{\mathbf a=0}$ as follows:
\begin{align}\label{eqn evaluation map}
x_{\mathbf a}|_{\mathbf a=0} = \sum_{w \in W^J} t_w \otimes \left( \sum_{i=1}^k b_{w,i} (0,\ldots, 0) u_i \right) \in \mathbb{H} \otimes_{\mathbb{H}_J} U .
\end{align}
\end{definition}

Let $w \in W^J$. Let $w=s_{\alpha_{r}}\ldots s_{\alpha_1}$ be a reduced expression of $w$. Let $R(w)=\left\{ \alpha \in R^+: w(\alpha)<0 \right\}$.
Define the intertwining element:
\begin{align} \label{eqn tilde tau} \widetilde{\tau}_w= (t_{s_{\alpha_r}}\alpha_r-k_{\alpha_r})\ldots (t_{s_{\alpha_1}}\alpha_1-k_{\alpha_1}) ,                   
\end{align}
\[ \tau_{w} =(t_{s_{\alpha_r}}\alpha_r-k_{\alpha_r})\ldots (t_{s_{\alpha_1}}\alpha_1-k_{\alpha_1}) \left(\prod_{\alpha \in R(w)} \alpha^{-1} \right) \in \mathcal H_J .\]
The way of normalization for $\tau_w$ will become clear in Proposition \ref{prop intertwin properties}. The well-definedness of $\tau_w$ follows from the following result:

\begin{proposition}\cite[Proposition 2.5(e)]{KR} \label{prop indep cho base}
For $w \in W^J$, $\tau_w$ is independent of the choice of a reduced expression.
\end{proposition}

\begin{proof}
Note that there is an assumption in \cite[Proposition 2.5(e)]{KR} but the proof still applies.
\end{proof}

\begin{lemma} \label{lem basic comm lem}
Let $w \in W^J$ and let $v \in V \subset \mathcal O(J)$. Then $v\tau_w=\tau_w w^{-1}(v)$.
\end{lemma}

\begin{proof}
It suffices to verify that for $v \in V$, $v(t_{s_{\alpha}}\alpha-k_{\alpha})=(t_{s_{\alpha}}\alpha-k_{\alpha})s_{\alpha}(v)$, which follows from $k_{\alpha}(v- s_{\alpha}(v))=k_{\alpha} \langle v, \alpha^{\vee}\rangle \alpha$ and a commutation relation of the graded Hecke algebra.
\end{proof}

\begin{proposition} \label{prop intertwin basis}
Let $x \in \mathbb{H} \otimes_{\mathbb{H}_J}U$. Then there exists a holomorphic element $x_{\mathbf a}$ in $\mathcal H^J \otimes_{\mathcal H_J} U_{\mathbf a}$ of the form
\[ \sum_{w \in W^J} \tau_w \left(\sum_{i=1} ^{N_w}q_{w,i} \otimes u_{w,i} \right)     \quad (q_{w,i} \in \mathcal O(J),\ u_{w,i} \in U, N_w \in \mathbb{Z}) \]
such that $x_{\mathbf a}|_{\mathbf a=0}=x$.
\end{proposition}

\begin{proof}
Any element in $\mathbb{H} \otimes_{\mathbb{H}_J} U$ can be written into the form of $\sum_{w \in W^J} t_w \otimes u_w$ The statement then follows from the fact that $\tau_w$ forms an  $\mathcal O(J)$-basis for $\mathcal H_J$. 
\end{proof}

\begin{remark}
One advantage to have such expression as in Proposition \ref{prop intertwin basis} is the nice commutation relation with the subalgebra $S(V)$. A drawback is hard to obtain certain uniqueness statement for such expression. For example, for $R=\left\{ \alpha \right\}$ of type $A_1$, let $U=\mathbb{C}u$ be the $1$-dimensional $S(V)$-representation with the weight $0$. We have $(\tau_{s_{\alpha}}\alpha+ k_{\alpha}) \otimes u|_{\mathbf a=0} =0$.

\end{remark}

Let $U$ be a finite-dimensional $\mathbb{H}_J$-module. For $\gamma^{\vee} \in V^{\vee}$, define 
\[ W(J, U, \gamma^{\vee}) = \left\{ (w, \lambda^{\vee}) \in W^J \times \mathrm{Wgt}(U) : w(\lambda^{\vee})=\gamma^{\vee}  \right\}.\]

\begin{proposition} \label{prop wv limit form}
Let $U$ be a finite-dimensional $\mathbb{H}_J$-module. Let $\gamma^{\vee} \in V^{\vee}$. Let $x \in \mathbb{H} \otimes_{\mathbb{H}_J} U$ be a generalized weight vector with weight $\gamma^{\vee}$. Then there exists a holomorphic element $x_{\mathbf a} \in \mathcal H^J \otimes_{\mathcal H_J} U_{\mathbf a}$ of the form
 \[    \sum_{(w, \lambda^{\vee}) \in  W(J, U,\gamma^{\vee})} \tau_w \left(\sum_{i=1}^{N_{w, \lambda^{\vee}}}q_{w,\lambda^{\vee}, i} \otimes u_{w,\lambda^{\vee}, i} \right) \quad (q_{w,\lambda^{\vee}, i} \in \mathcal O(J),\ u_{w,\lambda^{\vee}, i} \in 1 \otimes U \subset U_{\mathbf a}, N_{w, \lambda^{\vee}} \in \mathbb{Z}) \]
such that $x_{\mathbf a}|_{\mathbf a=0}=x$ and each $u_{w,\lambda^{\vee}, i}$ is a generalized weight vectors with the weight $\lambda^{\vee}$.
\end{proposition}

Since a complete proof for Proposition \ref{prop wv limit form} is lengthy, we separate out into an appendix. The idea of our proof is to construct weight vectors via an induction of the length of Weyl groups elements.



\subsection{Generalized standard modules} \label{ss gsms}
In this section, we shall refine our intertwining operator to the case of $(J,U) \in \Xi_L^g$. Recall that $W^J$ is the set consisting of all the minimal representatives for $W/ W_J$. Let $w_{0,J}$ be the longest element in $W_J$ and let $w^J$ be the longest element in $W^J$.

 We now consider two involutions. The first one is the map $\theta$ given by $\theta(\alpha)=-w_0(\alpha)$ ($\alpha \in J$). The second one is the map $\theta_J$ given by $\theta_J(\alpha)=-w_{0,J}(\alpha)$ ($\alpha \in J$). We also define $\phi_J=\theta \circ \theta_J$, which is not an involution in general. Since we only work for one fixed $J$, we shall simply write $\phi$ for $\phi_J$ most of time.

\begin{lemma} \label{lem bijective J}
Let $J \subset \Pi$. Then
\begin{enumerate}
\item the map $w \mapsto ww^J$ from $W^{\theta(J)}$ to $W^J$ is a bijective well-defined function.
\item For $w \in W^{\theta(J)}$, $l(w)+l(ww^J)=l(w^J)$.
\end{enumerate}
\end{lemma}
\begin{proof}
For (1), we first show that the map is well-defined i.e. $ww^J \in W^J$ for $w \in W^{\theta(J)}$. It is equivalent to show that $ww^J(\alpha)>0$ for any $\alpha \in J$. Since $w^J=w_0w_{0,J}$, $w^J$ sends simple roots in $J$ to simple roots in $\theta(J)$, which implies $ww^J(\alpha)>0$ for $\alpha \in J$. To show the map is bijective, we shall show that the map $w \mapsto ww^{\theta(J)}$ from $W^J$ to $W^{\theta(J)}$ gives the inverse map. This indeed follows from the following equations:
\[ w^Jw^{\theta(J)}=w^Jw^{\theta(J)}w_{0,\theta(J)}w_{0,\theta(J)}=w^Jw_0w_{0,\theta(J)}=w^Jw_{0,J}w_0=\mathrm{Id} .\] 
This proves (1).

For (2), first we have $l(ww^Jw_{0,J})  = l(ww_0)  = l(w_0)-l(w)$. On the other hand, we have
\begin{align*}
l(ww^Jw_{0,J}) & =l(ww^J)+l(w_{0,J}) \quad \mbox{ by (1) } 
\end{align*}
Now (2) is obtained by combining two equations.
\end{proof}

\begin{definition} \label{def delta map 1}
Let $\delta:V \rightarrow V$ be a linear isomorphism such that $\delta(J) \subset \Pi$. The map $\delta$ induces a linear isomorphism $\delta: V^{\vee} \rightarrow V^{\vee}$ such that $\delta^{-1}(\gamma^{\vee})(v)=\gamma^{\vee}(\delta(v))$.  Then $\delta$ induces a map $\widetilde{\delta}: \mathbb{H}_J \rightarrow \mathbb{H}_{\delta(J)}$ given by $\widetilde{\delta}(t_{s_{\alpha}})=t_{s_{\delta(\alpha)}}$, and $\widetilde{\delta}(v)=\delta(v)$ for $v \in V$. It is straightforward to verify that $\widetilde{\delta}$ is an algebra map. We shall simply write $\delta$ for $\widetilde{\delta}$ later. The map $\delta$ can also be similarly extended to $\mathcal H_J$.

Let $U$ be an $\mathbb{H}_J$-module. Define $\delta(U)$ to be the $\mathbb{H}_{\delta(J)}$-module such that $U$ is isomorphic to $\delta(U)$ as vector spaces via a map still denoted $\delta_U: U \rightarrow \delta(U)$ and the $\mathbb{H}_{\delta(J)}$-module is determined by
\[   \pi_{\delta(U)}(h) \delta_U(u) =\delta_U(\pi_{U}(\delta^{-1}(h))u) .\]
For lightening the notation, we simply write $\delta$ if the meaning is clear from the context.
\end{definition}

\begin{lemma} \label{lem singleton imply exist}
Let $(J, U) \in \Xi_L^g$. Let $\gamma^{\vee}$ be a weight of $U$. Then
\begin{enumerate}
\item $W(J, U, \gamma^{\vee})=\left\{(1, \gamma^{\vee}) \right\}$ and
\item $W(\theta(J), \phi(U), \gamma^{\vee})=\left\{ (w^{\theta(J)}, w^J(\gamma^{\vee})) \right\}$.
\end{enumerate}
\end{lemma}

\begin{proof}
Note that $W(J, U, \gamma^{\vee})$ is the set of weights of $I(J,U)$. Then (1) follows from Lemma \ref{lem decompose for lang class} (also see the proof of \cite{Ev} and \cite[Theorem 2.4]{KR}). 

For (2), by definitions, $w^J(\gamma^{\vee})=\phi(\gamma^{\vee})$ is a weight of $\phi(U)$. Since $w^{\theta(J)}(w^{J}(\gamma^{\vee}))=\gamma^{\vee}$ by Lemma \ref{lem bijective J}, $(w^{\theta(J)}, w^J(\gamma^{\vee})) \in W(\theta(J), \phi(U), -w_0(\gamma^{\vee}))$. To prove another inclusion, by Lemma \ref{lem bijective J}(1), it is equivalent to show that if there exists $w \in W^{\theta(J)}$ such that $w(-w_0(\gamma^{\vee}_1))=-w_0(\gamma^{\vee})$ for some weight $\gamma^{\vee}_1$ of $U$, then $w=1$. This is essentially the same as the proof of (1).
\end{proof}

We also define an analog for an $\mathbb H_{\delta(J)}$-module $\delta(U_{\mathbf a})$. 

\begin{definition}\label{def ua twist}
Define $\delta(U_{\mathbf{a}})$ to be an $\mathcal{H}_J$-module such that 
$\delta(U_{\mathbf{a}})$ is isomorphic to $\mathbb{C}(\mathbf a_1,\ldots,\mathbf a_r) \otimes_{\mathbb{C}} \delta(U)$ as vector spaces and the action of $\mathcal H_J$ is determined by
\[ \pi_{\delta(U_{\mathbf{a}}) }(h)\delta_{U_{\mathbf{a}}}(u_{\mathbf{a}}) =\delta_{U_{\mathbf{a}}}(\pi_{U_{\mathbf{a}} }(\delta^{-1}(h))u_{\mathbf{a}}) 
\]
Again, we shall simply write $\delta$ for $\delta_{U_{\mathbf{a}}}$ if there is no confusion. 

Recall $\theta$ is defined in the beginning of this section. We define the evaluation $|_{\mathbf a=0}$ for holomorphic elements of $\mathcal H^{\theta(J)} \otimes_{\mathcal H_{\theta(J)}} \phi(U_{\mathbf a})$ as (\ref{eqn evaluation map}) for $\mathcal H^J \otimes_{\mathcal H_J} U_{\mathbf a}$.

\end{definition}



We also remark that it does not really make sense to write $\phi(U)_{\mathbf a}$ since the definition of $U_{\mathbf a}$ depends on a choice of basis.

We now use the intertwining element $\tau_{w^{\theta(J)}}$ to define an intertwining operator from $\mathcal H^J \otimes_{\mathcal H_J} U_{\mathbf a}$ to $\mathcal H^{\theta(J)} \otimes_{\mathcal H_{\theta(J)}} \phi(U_{\mathbf a})$. One may consider as an analogue of a Knapp-Stein intertwining operator  (\cite{KS}, also see \cite[Section 14]{ALTV}). There are similar results for standard modules in the affine Hecke algebra setting in an unpublished work of Delorme-Opdam (also see relevant work in \cite{DO2}). 

We also remark that $(\theta(J), \phi(U))$ is not in $\Xi_L^g$ in general.

\begin{proposition} \label{prop intertwin properties}
Let $(J, U) \in \Xi_L^g$. Then:
\begin{enumerate}
\item  Any element $   \tau_{w^{\theta(J)}} \otimes \phi(u) \in \mathcal H^{\theta(J)} \otimes_{\mathcal H_{\theta(J)}} \phi(U_{\mathbf a})$ with $u \in 1 \otimes U \subset U_{\mathbf a}$ is holomorphic.
\item  The subspace
 \[\left\{  \tau_{w^{\theta(J)}} \otimes \phi(u_{\mathbf a}) : u_{\mathbf a} \in U_{\mathbf a} \right\} \] of $\mathcal{H}^{\theta(J)} \otimes_{\mathcal{H}_{\theta(J)}} U_{\mathbf a}$ is invariant under the $\mathbb{H}_{J}$-action and is isomorphic to $U_{\mathbf{a}}$ as an $\mathbb H_{J}$-module. The isomorphism is characterized by the map:
\[ 1 \otimes u_{\mathbf a} \mapsto \tau_{w^{\theta(J)}} \otimes \phi(u_{\mathbf a})   
\]
for $u \in 1 \otimes U \subset U_{\mathbf a}$. 
\item The isomorphism in (2) induces an $\mathbb{H}$-module isomorphism from $\mathcal{H}^J\otimes_{\mathcal{H}_J} U_{\mathbf a}$ to $\mathcal H^{\theta(J)} \otimes_{\mathcal{H}_{\theta(J)}} \phi(U_{\mathbf a})$
\item  The subspace
 \[\left\{ ( \tau_{w^{\theta(J)}} \otimes \phi(u))|_{\mathbf a=0} : u \in 1 \otimes U \subset U_{\mathbf a} \right\} \] of $\mathbb{H} \otimes_{\mathbb{H}_{\theta(J)}} \phi(U)$ is invariant under the $\mathbb{H}_{J}$-action and is isomorphic to $U$ as $\mathbb{H}_{J}$-modules via the following map:
\[  1\otimes u \mapsto (\tau_{w^{\theta(J)}} \otimes \phi(u))|_{\mathbf a=0}  .\]

\item The map in (4) induces an $\mathbb{H}$-map $ I(J, U) \rightarrow I(\theta(J), \phi(U))$. Moreover, if $(J,U) \in \Xi_L$, then the image of the map is isomorphic to the unique simple quotient $L(J, U)$ of $I(J,U)$. 
\end{enumerate}
\end{proposition} 

\begin{proof}
Let $\gamma^{\vee}$ be a weight of $U$. Fix a set of generalized $\gamma^{\vee}$-weight vectors $\left\{ u_1, \ldots, u_k \right\}$ in $U$. It is not hard to show from linear algebra that there exists a generalized weight vector $x$ with the weight $w^{\theta(J)}(\phi(\gamma^{\vee}))=\gamma^{\vee}$ of the form 
\[   x = t_{w^{\theta(J)}} \otimes \phi(u_{l}) + \sum_{w \in W^{\theta(J)} \setminus \left\{ w^{\theta(J)} \right\}} t_w \otimes \phi(u_w) ,
\]
where $u_w \in U$. By Proposition \ref{prop wv limit form} and Lemma \ref{lem singleton imply exist}, there exists a holomorphic element $x_{\mathbf a} \in \mathcal{H}^J\otimes_{\mathcal{H}_J} U_{\mathbf a}$ of the form
\[ x_{\mathbf a} = \sum_{i=1}^k \tau_{w^{\theta(J)}}  \otimes b_i \phi(u_{i})
\]
such that $x_{\mathbf a}|_{\mathbf a=0}=x$. By considering the term $t_{w^{\theta(J)}} \otimes \phi(u_i)$, we have $q_i$ is holomorphic for each $i$. Furthermore $b_l(0,\ldots, 0)\neq 0$. Hence we also have $b_l^{-1}$ is holomorphic at $\mathbf a=0$. Thus $\tau_w \otimes u_i=\tau_w \otimes b_l^{-1} (b_l u_l)$ is holomorphic for each $l$. Then by linearity, we obtain (1).

We now consider (2). For notation simplicity, set $w=w^{\theta(J)}$ For any $v \in V$, we have $v \tau_{w} = \tau_{w} w^{-1}(v)$ (Lemma \ref{lem basic comm lem}). For any $\alpha \in J$, $s_{\alpha}w \notin W^J$ and $l(s_{\alpha}w)=l(w)+1$ (see the proof of Lemma \ref{lem bijective J}) and hence Lemma \ref{lem various l case a}(2) in Appendix B implies $t_{s_{\alpha}}\tau_{w} =\tau_{w} t_{s_{w^{-1}(\alpha)}}$. By $\phi^{-1}(w^{-1}(v))=v$, we can verify (2).

For (3), we see the induced map sends $\tau_{w^J} \otimes u$ to $q \otimes  u$ for some invertible $q \in \mathcal O(J) $ and $u \in U$. Explicitly,
\[ q = \prod_{\alpha \in R^+ \setminus R^+_J } \left( \frac{\alpha^2 - k_{\alpha}}{\alpha^2} \right).\]
Hence the induced map is invertible.

For (4), the map is well-defined by (1). It follows from (2) that the map is an $\mathbb{H}$-map.

The first assertion for (5) follows directly from (4) and Frobenius reciprocity. For the second assertion, in a paper of Barbasch-Moy \cite{BM2}, there is a notion of anti-involution $*$ on $\mathbb{H}$. (Here we shall use a linear version for $*$ rather than Hermitian-linear version.) Hence $*$ determines a dual module $I(\theta(J), \phi(U))^*$ and it is shown in  \cite[Corollary 1.3]{BM2} that $I(\theta(J), \phi(U))^* \cong I(\theta(J), \phi(U)^{*_{\theta(J)}})$. By a slight consideration on the weights of $\phi(U)^{*_{\theta(J)}}$, we have $(\theta(J),\phi(U)^{*_{\theta(J)}}) \in \Xi_L$. From the constructions in Langlands classification, $L(\theta(J), \theta(U))$ is the unique composition factor $X$ of $I(\theta(J), \theta(U))$ which $\mathrm{Res}_{\mathbb{H}_J}X$ contains the composition factor of $\phi(U)^{*_{\theta(J)}}$. Then $L(\theta(J), \phi(U)^{*_{\theta(J)}})^*$ is the unique composition factor $X$ of $I(\theta(J), \phi(U))$ which $\mathrm{Res}_{\mathbb{H}_{\theta(J)}}X$ contains the composition factor $\phi(U)$. Hence the map in (5) (which is non-zero) is determined (up to scalar). Hence the image of the map is the unique simple submodule of $I(\theta(J),\phi(U))$ and so has to be isomorphic to the unique simple quotient of $I(J,U)$. 
\end{proof}

Let $(J,U) \in \Xi^g_L$. Denote the isomorphism in Proposition \ref{prop intertwin properties}(3) by $\Delta_{\mathbf a}^U : \mathcal H^J \otimes_{\mathcal H_J} U_{\mathbf a} \rightarrow \mathcal H^{\theta(J)} \otimes_{\mathcal H_{\theta(J)}} \phi(U_{\mathbf a})$. Denote by $\Delta^U$ the map in (5). For simplicity, we shall drop the superscripts if there is no confusion.

\begin{definition} \label{def quotient gen}
For $(J,U) \in \Xi_L^g$, define $L(J,U)$ to be the image of $\Delta^U$. This coincides with the earlier notation of $L(J,U)$ for $(J,U) \in \Xi_L$ (up to an isomorphism) by Proposition \ref{prop intertwin properties}(5).
\end{definition}

\begin{remark}
In general, for $(J,U) \in \Xi_L^g \setminus \Xi_L$, the space of intertwining operators from $\mathbb{H} \otimes_{\mathbb{H}_J} U$ to $\mathbb{H} \otimes_{\mathbb{H}_{\theta(J)}}\phi(U)$ has dimension greater than $1$.
\end{remark}




\subsection{Intertwining operator along arbitrary directions} \label{ss intertwin direction}
We keep using notations in previous subsections. Let $\mathbf t$ be an indeterminate and let $\mathbb{C}(\mathbf t)$ be the rational function ring over $\mathbf t$. Let $(J,U) \in \Xi^g_L$.  Let $\eta^{\vee} \in V_J^{\vee}$ (possibly zero). Let $\mathbf t$ be an indeterminate. Let $U_{\mathbf t \eta^{\vee}}=\mathbb{C}(\mathbf t) \otimes_{\mathbb C}U$. We shall consider $U_{\mathbf t\eta^{\vee}}$ as an $\mathbb{H}_J$-module over $\mathbb{C}(\mathbf t)$ such that the action is given by
\[    \pi_{U_{\mathbf t\eta^{\vee}}}(t_w)(b\otimes u)= b \otimes \pi_{U}(t_w) u \quad \mbox{ for $w \in W$ }
\]
\[   \pi_{U_{\mathbf t\eta^{\vee}}}(v)(b\otimes u)= b \otimes \pi_{U}(v) u + \mathbf t\eta^{\vee}(v) b \otimes u \quad \mbox{ for $v\in V$ },
\]
where $b \in \mathbb{C}(\mathbf t)$. We define $\phi(U_{\mathbf t\eta^{\vee}})$ analogous to the notion of $\phi(U_{\mathbf a})$ in Definition \ref{def ua twist}. 

We consider an element $1 \otimes u \in \mathbb{H}\otimes_{\mathbb{H}_J} U_{\mathbf t \eta}$, where $u \in 1\otimes U \subset U_{\mathbf t \eta^{\vee}}$. The element $1 \otimes u$ is also naturally inside $1 \otimes U \subset U_{\mathbf a}$ and by Proposition \ref{prop intertwin properties}(1), $\tau_{w^{\theta(J)}} \otimes \phi(u) \in \phi(U_{\mathbf a})$  is an holomorphic element and so we can specialize $\tau_{w^{\theta(J)}} \otimes \phi(u) \in \mathcal H^{\theta(J)} \otimes_{\mathcal H_{\theta(J)}} \phi(U_{\mathbf a})$ at $\mathbf a=\mathbf t \eta^{\vee}$ for small $\mathbf t$. (Here the precise meaning of $\mathbf a=\mathbf t\eta^{\vee}$ is to specialize $\eta^{\vee}_{\mathbf a}=\mathbf t\eta^{\vee}$.) Now the element $\tau_{w^{\theta(J)}} \otimes u|_{\mathbf a=\mathbf t\eta^{\vee}}$ is in  $\mathbb{H}\otimes_{\mathbb{H}_{\theta(J)}}\phi(U_{\mathbf t\eta^{\vee}})$. We now define 
\begin{align}\label{eqn map intertwin}
\Delta_{\mathbf t\eta^{\vee}}^U: \mathbb{H} \otimes_{\mathbb{H}_J} U_{\mathbf t\eta^{\vee}} \rightarrow \mathbb{H} \otimes_{\mathbb{H}_{\theta(J)}} \phi(U_{\mathbf t \eta^{\vee}}) 
\end{align}
to be the $\mathbb{H}$-map extending 
\[ 1 \otimes u \mapsto (\tau_{w^{\theta(J)}} \otimes \phi(u))|_{\mathbf a=\mathbf t\eta^{\vee}} ,\]
where $u \in 1 \otimes U \subset U_{\mathbf t\eta^{\vee}}$. Again, we may simply write $\Delta_{\mathbf t\eta^{\vee}}$ or $\Delta_{\mathbf t}$ if there is no confusion.

\begin{remark}
For $u \in 1 \otimes U \subset U_{\mathbf t\eta^{\vee}}$, the notion $\tau_{w^{\theta(J)}} \otimes u$ is well-defined for generic $\eta^{\vee}$ but not for all $\eta^{\vee} \in V^{\vee}_J$. Hence we have to pass to $U_{\mathbf a}$ to construct the intertwining operator.
\end{remark}

\section{Jantzen filtrations} \label{s jan filt}

One may also compare the way to define a Jantzen filtration of this section with the paper \cite{BC}. Those two Jantzen filtrations (for the case of standard modules) coincide by the uniqueness of the intertwining operator. We shall not need this fact except in the computations of Example \ref{ex C3} which assume the truth of a version of Jantzen conjecture.

\subsection{Jantzen filtrations} \label{ss JF}

 The Jantzen filtration is defined through $\Delta_{\mathbf t\eta^{\vee}}$ in Section \ref{ss intertwin direction}. One may define a multivariable version of Jantzen filtration through $\Delta_{\mathbf a}$. However, if we hope to have some way to compute the Jantzen filtration from  geometric way (e.g. a conjecture in \cite{BC} using \cite{Lu2, Lu3}) or from Arakawa-Suzuki functor \cite{AS, Su} or results in \cite{Ro}, then perhaps a single variable is a better notion.

\begin{definition}(Jantzen filtration \cite{Ja}) 
Let $(J,U) \in \Xi_L^g$. Let $\eta^{\vee} \in   V_J^{\vee, \bot}$, which is possibly zero. Set $\mathbf a= \mathbf t \eta^{\vee}$. Let $\Omega(\mathbf{t}) \subset \mathbb{C}(\mathbf{t})$ be the set of all functions in $\mathbb{C}(\mathbf t)$ holomorphic at $\mathbf{t}=0$. We can define a Jantzen filtration as follows. For each integer $i$, let 
\[ U^i_{\mathbf{t}\eta^{\vee}}=  \mathbf{t}^i  \Omega(\mathbf{t}) \otimes U \subset U_{\mathbf t\eta^{\vee}}.\] 
 We shall regard $U^i_{\mathbf{t}\eta^{\vee}}$ as an $\mathbb{H}_J$-module.
Recall that $\phi=\theta \circ \theta_J$. Set $T^i(U_{\mathbf t\eta^{\vee}})=\mathbb{H}\otimes_{\mathbb{H}_{\theta(J)}}\phi(U_{\mathbf t \eta^{\vee}}^i)$ and  $T^i(\phi(U_{\mathbf t \eta^{\vee}}^i))=\mathbb{H}\otimes_{\mathbb{H}_{\theta(J)}}\phi(U_{\mathbf t \eta^{\vee}})$. From this point, we shall consider $T^i(U_{\mathbf t\eta^{\vee}})$ and $T^i(\phi(U_{\mathbf t \eta^{\vee}}^i))$ to be $\mathbb{H}$-modules over $\mathbb C$, but the $\mathbb{H}$-action is not significant sometimes and we may simply write $T^i(U)$ and $T^i(\phi(U))$ respectively to lighten notations.

 Via the specialization at $\mathbf t=0$, there is a natural identification
\begin{align} \label{eqn isom quotient}   
 T^0(U_{\mathbf t \eta^{\vee}})/ T^1(U_{\mathbf t \eta^{\vee}}) \stackrel{\sim}{\longrightarrow}\mathbb{H} \otimes_{\mathbb{H}_J} U  .
\end{align}
The Jantzen filtration is, roughly speaking, to get the information of the vanishing degree of elements in $\mathbb{H} \otimes_{\mathbb{H}_J}U$ under the intertwining operator.

By Proposition \ref{prop intertwin properties} (1), we have the map 
\[ \Delta_{\mathbf t\eta^{\vee}}^{i,U}=pr_i \circ \Delta^U_{\mathbf t\eta^{\vee}}: T^0(U) \rightarrow T^0(\phi(U))/T^i(\phi(U)),
\]
where $\Delta_{\mathbf t\eta^{\vee}}^{U}$ is the map in (\ref{eqn map intertwin}) restricted to the space $T^0(U_{\mathbf{t}\eta^{\vee}})$ and $pr_i$ is the natural projection map from $T^0(\phi(U))$ to $T^0(\phi(U))/T^i(\phi(U))$. We may drop the superscript $U$ and write $\Delta_{\mathbf t\eta^{\vee}}^i$ sometimes.


We then have the filtration of the form:
\[ F^i_{\eta^{\vee}}(J, U) = (\ker \Delta_{\mathbf t\eta^{\vee}}^{i} + T^1(U))/T^1(U) .
\]

By (\ref{eqn isom quotient}), this gives a filtration, denoted $\mathrm{JF}^i_{\eta^{\vee}}(J,U)$ or simply $\mathrm{JF}^i_{\eta^{\vee}}$, for the generalized standard module $\mathbb{H} \otimes_{\mathbb{H}_{J}} U$:
\[  \mathbb{H} \otimes_{\mathbb{H}_{J}}U =\mathrm{JF}^0_{\eta^{\vee}}(J,U) \supseteq \mathrm{JF}^1_{\eta^{\vee}}(J,U) \supseteq \mathrm{JF}^2_{\eta^{\vee}}(J,U) \supseteq \ldots 
\]
We shall call the filtration to be the {\it Jantzen filtration of $I(J,U)$ associated to $\eta^{\vee}$}. A more intuitive way to describe $\mathrm{JF}^i_{\eta^{\vee}}(J,U)$ via the identification (\ref{eqn isom quotient})  is as follows. The space $\mathrm{JF}^i_{\eta^{\vee}}(J,U)$ contains elements $x \in \mathbb{H} \otimes_{\mathbb{H}_J} U$ such that there exists an element $x_{\mathbf t}\in \mathbb{H} \otimes_{\mathbb{H}_J} U_{\mathbf t\eta^{\vee}}$ satisfying the conditions that $x_{\mathbf t}|_{\mathbf t=0}=x$ and $\Delta_{\mathbf t\eta^{\vee}}(x_{\mathbf t})$ has zero of order at least $i$. 

From the definitions, we have 
\[  \mathrm{JF}^0_{\eta^{\vee}}(J,U)/\mathrm{JF}^1_{\eta^{\vee}}(J,U) \cong L(J,U) \]
for any $\eta^{\vee}  \in V_J^{\vee, \bot}$.


Note that our definition allows $\eta^{\vee}$ to be in arbitrary direction, and possibly zero. In general, we do not have $\bigcap_i \mathrm{JF}^i_{\eta^{\vee}} = 0$. When $\eta^{\vee}=0$, we have $\mathrm{JF}^0_{\eta^{\vee}}(J,U)/\mathrm{JF}^i_{\eta^{\vee}}(J,U) \cong L(J,U)$ for all $i \geq 1$.
\end{definition}


\subsection{Linearly independence} \label{ss line inde lemma}
We keep using notations in the previous subsection. For each integer $i$ and each $x \in F^i(J,U_{\mathbf t \eta^{\vee}}) \setminus F^{i+1}(J,U_{\mathbf t \eta^{\vee}})$, let $x_{\mathbf t} \in \ker \Delta_{\mathbf t\eta^{\vee}}^{i}$ be a representative of $x$. In particular, $\Delta_{\mathbf t\eta^{\vee}}^{i}(x_{\mathbf t})$ has zeros of order $i$, i.e.
\[ \frac{1}{\mathbf t^i}\Delta_{\mathbf t\eta^{\vee}}(x_{\mathbf t}) \in T^0(\phi(U )) \quad \mbox{ and } \quad \frac{1}{\mathbf t^{i+1}}\Delta_{\mathbf t\eta^{\vee}}(x_{\mathbf t}) \notin T^0(\phi(U ))  \]

For each $i$, fix the representatives $x^{i,1}_{\mathbf t}, \ldots, x^{i,d_i}_{\mathbf t} \in T^0(U)$ for $F^i(J,U_{\mathbf t \eta^{\vee}}) \setminus F^{i+1}(J,U_{\mathbf t \eta^{\vee}})$ such that the projections of $x^{i,1}_{\mathbf t}, \ldots, x^{i,d_i}_{\mathbf t}$ form a basis for $ F^i(J,U_{\mathbf t \eta^{\vee}}) / F^{i+1}(J,U_{\mathbf t \eta^{\vee}})$. Let $\left\{ x^{\infty,k}_{\mathbf t} \right\}_{k=1,\ldots d_{\infty}} \subset T^0(U)$ whose projection in $(\ker \Delta_{\mathbf t\eta^{\vee}}+T^1(U))/T^1(U)$ forms a basis. By definition $\Delta_{\mathbf t \eta^{\vee}}^{i}(x_{\mathbf t}^{\infty,k})=0$ for all $i$. 

Let $X^i_{\mathbf t}(J, U)$ (or simply $X^i_{\mathbf t}$) be the space spanned by $x^{i,1}_{\mathbf t}, \ldots, x^{i,d_i}_{\mathbf t}$. Let $X_{\mathbf t}$ be the space spanned by all $X^i_{\mathbf t}$ (which is not an $\mathbb{H}$-module in general). From definitions, we have 
\[ X_{\mathbf t} \cap T^1(U) =0 .\]

We now prove a statement of linearly independence.
\begin{proposition} \label{prop strong linearly indpende}
The set
 \[\left\{ \left. \frac{1}{\mathbf t^i}\Delta_{\mathbf t\eta^{\vee}} (x_{\mathbf t}^{i, 1})\right|_{\mathbf t=0},\ldots,\left. \frac{1}{\mathbf t^i} \Delta_{\mathbf t\eta^{\vee}}(x^{i,d_i}_{\mathbf t})\right|_{\mathbf t=0} \right\}_{i\in \mathbb{Z}  } \]
 is linearly independent. Equivalently,
 \[ \mathrm{span}\left\{ \frac{1}{\mathbf t^i}\Delta_{\mathbf t\eta^{\vee}} (x_{\mathbf t}^{i, 1}),\ldots, \frac{1}{\mathbf t^i} \Delta_{\mathbf t\eta^{\vee}}(x^{i,d_i}_{\mathbf t}) \right\}_{i\in \mathbb{Z}  } \cap T^1(\phi(U))=0 \]
\end{proposition}

\begin{proof}
We consider a linear equation of the following form
\[ \sum_{i \in \mathbb{Z}} \frac{1}{\mathbf t^i} \sum_{k=1}^{d_i} a_{i,k}  \Delta_{\mathbf t\eta^{\vee}} (x_{\mathbf t}^{i, k})  =y ,
\]
where $y \in T^1(\phi(U_{\mathbf{t}\eta^{\vee}}))$. Let $i'$ be the greatest integer such that $a_{i',k} \neq 0$ for some $k=1, \ldots, d_{i'}$. Suppose such $i'$ exists. 

\[   \sum_{k=1}^{d_{i'}} a_{i',k} \Delta_{\mathbf t\eta^{\vee}} (x_{\mathbf t}^{i', k})  \in  T^{i'+1}(\phi(U)) \]
Therefore $ \sum_{k=1}^{d_{i'}} a_{i',k}\Delta^{i'+1}_{\mathbf t\eta^{\vee}}(x_{i',k}) =0 $. Hence $\sum_{k=1}^{d_{i'}} a_{i',k}x_{i',k} \in F^{i'+1}(J,U)$ and this gives a contradiction to our choice of $x_{i',k}$.
\end{proof}

\subsection{Comparing filtrations} \label{ss remark filt}

In general, the Jantzen filtration for a standard module does not coincide with the radical filtration or socle filtration. This gives a discrepancy from the picture of real groups \cite[Corollary 5.3.2]{BB} (also see \cite{Ba}, \cite{Ir} for Verma module cases). We shall give an example for type $B_2$. 

We recall some definitions. The radical of an $\mathbb{H}$-module $X$, denoted $\mathrm{rad}(X)$, is the minimal submodule of $X$ such that the quotient is semisimple. This defines the radical filtration for $X$:
\[ X \supset \mathrm{rad}^1(X) \supset \mathrm{rad}^2(X) \supset \ldots \supset 0 \]
The socle, denoted $\mathrm{soc}(X)$, of an $\mathbb{H}$-module $X$ is the maximal semisimple module of $X$. Define inductively $\mathrm{soc}^{i+1}(X)$ by $\mathrm{soc}^{i+1}(X)/\mathrm{soc}^i(X) \cong \mathrm{soc}(X/\mathrm{soc}^i(X))$. This gives the socle filtration for $X$:
\[ 0\subset \mathrm{soc}^1(X) \subset \mathrm{soc}^2(X) \subset \ldots \subset X 
\]

\begin{example}
We consider the case of type $B_2$ with $k_{\alpha}=k_{\beta}=2$. Denote the simple roots $\alpha$ and $\beta$ such that 
\[\alpha^{\vee}(\alpha)=\beta^{\vee}(\beta)=-\alpha^{\vee}(\beta)=2, \quad \beta^{\vee}(\alpha)=-1 . 
\]
We consider the central character $\gamma^{\vee}=\alpha^{\vee}$. There exists two non-isomorphic irreducible tempered module of the central character $\gamma^{\vee}$. Denote by $T_0$ the tempered module which contains a sign representation as $W$-representation and denote by $T_1$ the tempered module which is $1$-dimensional. There is another $1$-dimensional irreducible $\mathbb{H}$-module with the central character $\gamma^{\vee}$. Denote the module by $Z$. The weight space of $Z$ is $\alpha^{\vee}$.

Using the projective resolution in \cite[Section 3]{Ch2}, simple computations for $\mathrm{Hom}$-space of $W$-representations give
\[\dim \mathrm{Ext}_{\mathbb{H}}^i(Z, Z)= \dim \mathrm{Ext}_{\mathbb{H}}^i(T_1, T_1)=\left\{ \begin{array} {cc} 
                                           1   &   \mbox{ for $i=0$ }  \\
																					 0   &   \mbox{ for $i \neq 0$ }
																			\end{array} \right.
\]  
\begin{align}\label{eqn compute 1} \dim \mathrm{Ext}_{\mathbb{H}}^i(A , B)=\left\{ \begin{array} {cc} 
                                           1   &   \mbox{ for $i=1$ }  \\
																					 0   &   \mbox{ for $i \neq 1$ }
																			\end{array} \right. ,
\end{align}
where $(A, B)$ can be one of the following pairs: $(T_1, Z), (Z, T_1), (T_1, T_0), (T_0, T_1), (Y, Z), (Z, Y)$. 

To compute $\mathrm{Ext}$-groups for the pairs $(T_0, T_0)$ and $(Y, Y)$, using the projective resolution of \cite[Section 3]{Ch2} requires some further structural information for $T_0$ and $Y$. Instead of computing such information, we apply the duality result \cite[Theorem 4.15]{Ch2} and then obtain
\[ \dim\mathrm{Ext}^2_{\mathbb{H}}(Y, Y)=\dim\mathrm{Hom}_{\mathbb{H}}(T_0, Y)= 0 .\]
Then using the $W$-structure, we obtain the Euler-Poincare pairing $\mathrm{EP}(T_0,T_0)=\mathrm{EP}(Y, Y)=1$. Combining with the fact that the global dimension of $\mathbb{H}$ is $2$, we can deduce that
\begin{align} \label{eqn compute 2} \mathrm{Ext}^i_{\mathbb{H}}(Y, Y)=\mathrm{Ext}^i_{\mathbb{H}}(T_0,T_0)=\left\{ \begin{array} {cc} 
                                           1   &   \mbox{ for $i=0$ }  \\
																					 0   &   \mbox{ for $i \neq 0$ }
																			\end{array} \right.
\end{align}

Let $J=\left\{ \alpha \right\}$ and let $U$ be the irreducible (2-dimensional) tempered $\mathbb{H}_J$-module with the weight $0$. Let $\nu^{\vee}=\alpha^{\vee}+2\beta^{\vee} \in V_J^{\vee,\bot}$. We consider the standard module $X=\mathbb{H} \otimes_{\mathbb{H}_{J}} (U\otimes \mathbb{C}_{\nu^{\vee}})$, where $\mathbb{C}_{\nu^{\vee}}$ is $1$-dimensional $S(V_J^{\bot})$-module with the weight $\nu^{\vee}$. By a slight consideration on the weight space of $X$, one can deduce that $T_0$, $T_1$, $Y$ and $Z$ have multiplicity one in the composition series of $X$. 

The simple quotient of $X$ is $Y$ and hence we have the following short exact sequence:
\[   0 \rightarrow N \rightarrow X \rightarrow Y \rightarrow 0,
\]
where $N$ is the maximal proper submodule of $X$. 

In order to determine the radical $\mathrm{rad}(N)$ of $N$, we have to compute $\mathrm{Hom}_{\mathbb{H}}(N, C)$, where $C=Z, T_0, T_1$. By using Proposition \ref{prop ext lower langlands para} and the associated long exact sequence of the functor $\mathrm{Hom}_{\mathbb{H}}(.,C)$, we have
\[ \mathrm{Hom}_{\mathbb{H}}(N, C)\cong \mathrm{Ext}^1_{\mathbb{H}}(Y, C) . \]
and so the $\mathrm{Hom}$-space can be computed from (\ref{eqn compute 1}) and (\ref{eqn compute 2}). This implies there exists a surjective map from $N$ to $Z \oplus T_1$ and $\mathrm{rad}(N)=T_0$. In summary, the radical filtration of $X$ is given by: $\mathrm{rad}^0(X)/\mathrm{rad}^1(X)=Y$,  $\mathrm{rad}^1(X)/\mathrm{rad}^2(X)\cong T_1 \oplus Z$,  $\mathrm{rad}^2(X)/\mathrm{rad}^3(X)\cong T_0$.

On the other hand, the Jantzen filtration of $X$ can be computed directly (see Appendix A for the details, also {\it c.f.} \cite[Section 4.3]{Ci}, \cite[Section 6.7]{BC} ):
\[ \mathrm{JF}^0_{\nu^{\vee}}/\mathrm{JF}^1_{\nu^{\vee}} \cong Y,\ \mathrm{JF}^1_{\nu^{\vee}}/ \mathrm{JF}^2_{\nu^{\vee}} \cong T_1 
\]
\[ \mathrm{JF}^2_{\nu^{\vee}}/\mathrm{JF}^3_{\nu^{\vee}} \cong Z,\ \mathrm{JF}^3_{\nu^{\vee}}/ \mathrm{JF}^4_{\nu^{\vee}} \cong T_0 
\]

For the socle filtration, we apply the $\bullet$ anti-involution. By a weight consideration, we have $X^{\bullet}\cong IM(X)$, where $IM$ is the Iwahori-Matsumoto involution (see \cite{Ev} for the definition of $IM$). In particular, $IM(T_0)=Y$ and $IM(T_1)=Z$. The radical filtration of $X^{\bullet} \cong IM(X)$ and Lemma \ref{lem bullet dual ismo} now determine the socle filtration of $X$. In particular, $T_1$ is a tempered module which does not appear as a submodule in the standard module $X$.

\end{example}

\section{The quotient $L(J,U)$ of generalized standard modules} \label{s quotient gsm}

Recall that $S(V)$-extensions are defined in Definition \ref{def sv extension}. In this section, we deal with generalized standard modules of strict $S(V)$-type. Since those extensions come from extensions of representations of a polynomial ring, we can write down the structure of corresponding modules in a fairly explicit way.

However, in order to describe the structure in terms of composition factors, we need some more work and make use of Jantzen filtrations. 

We shall make use of notations in Sections \ref{s intertwin operat} and \ref{s jan filt} (e.g. $U_{\mathbf a}$, $\mathcal H_J$, $\Delta$, etc).
\subsection{Realization of $S(V)$-extensions}  \label{ss realization sv ext}

Let $(J,U) \in \Xi_L$. (Some constructions are also valid for $(J,U) \in \Xi_L^g$, but those are not our main concern.) Then $U=\overline{U} \otimes L$ for some $\mathbb{H}_J^{ss}$-module $\overline{U} $ and some one dimensional $S(V_J^{\bot})$-module $L$. From the discussions on Section \ref{s gss ext}, we can naturally construct Yoneda first extensions between two $I(J,U)$ through the Yoneda first extensions between two $U$. Using discussions in Section \ref{s gss ext}, we can and shall identify the following:
\[ \mathrm{Ext}^1_{\mathbb{H}}(I(J,U), I(J,U)) \cong \mathrm{Ext}^1_{\mathbb{H}}(U, U) \cong \mathrm{Ext}^1_{\mathbb{H}_J^{ss}}(\overline{U}, \overline{U})\oplus \mathrm{Ext}^1_{S(V_J^{\bot})}(L, L). 
\]

We identify $V_J^{\vee, \bot}$ with $\mathrm{Ext}^1_{\mathrm{S(V_J^{\bot})}}(L, L)$. Let $\eta^{\vee} \in V_J^{\bot} =\mathrm{Ext}^1_{S(V_J^{\bot})}(L, L)$. The element $\eta^{\vee}$ gives rise an indecomposable $r \times\dim U$-dimensional representation, denoted $U^{r, \eta^{\vee}}$ with the explicit structure described as follow. As vector spaces, we identify 
\begin{align} \label{eqn isom r copies}
U^{r, \eta^{\vee}} \cong U^{\oplus r}= U \oplus \ldots \oplus U,
\end{align}
 where $U$ appears $r$ times in the right hand side. The $\mathbb{H}_J$-action is given by:
\begin{align} \label{eqn action gst 1}  \pi_{U^{r, \eta^{\vee}}}(t_{w}) (u_1, \ldots, u_r) = (\pi_{U}(t_{w})u_1, \ldots, \pi_{U}(t_{w})u_r) \quad (w \in W),
\end{align}
\begin{align} \label{eqn action gst 2}  \pi_{U^{r, \eta^{\vee}}}(v)(u_1, \ldots, u_r)=(\pi_{U}(v)u_1,\ldots, \pi_U(v)u_r)+\eta^{\vee}(v)(0, u_1 \ldots, u_{r-1}) \quad (v \in V) .
\end{align}
By definitions, $(J, U^{r, \eta^{\vee}}) \in \Xi^g_L$. 

\begin{lemma} \label{lem bullet dual}
Let $(J,U) \in \Xi_L$. Then $I(J,U^{r,\eta^{\vee}})^{\bullet} \cong I(\theta(J), \phi(U^{r,\eta^{\vee}}))$. 
\end{lemma}

\begin{proof}
Note that $I(J,U^{r,\eta^{\vee}})^{\bullet}\cong \theta(I(J,U^{r,\eta^{\vee}})^{*})$ (see e.g. \cite[Lemma 4.5]{Ch2}). Here $*$ is a linear anti-involution defined similarly as in \cite[Section 1]{BM2}. Then by \cite[Corollary 1.3]{BM2}, $I(J,U^{r,\eta^{\vee}})^{*}\cong I(\theta(J), (U^{r, \eta^{\vee}})^{*_J})$. Now by considering weights of $U^{*_J}$ and using \cite[Theorem 5.5]{EM}, we have $ U^{*_J} \cong \theta_J(U)$. Then $(U^{r, \eta^{\vee}})^{*_J} \cong \theta_J(U^{r,\eta^{\vee}})$ (which we also need to use the strict $S(V)$-extensions for $U^{r,\eta^{\vee}}$). Hence $I(J,U^{r,\eta^{\vee}})^{\bullet}\cong I(\theta(J),\phi(U^{r,\eta^{\vee}}))$.
\end{proof}

Note that $I(J,U^{r,\eta^{\vee}})$ can be naturally identified with $T^0(U)/T^r(U)$ via the map
\[  t_w \otimes (u_1, \ldots, u_r) \mapsto \sum_{i=1}^r t_w \otimes \mathbf t^{i-1}u_i ,\]
where $u_i$ on the left-hand side is also regarded as an element in $1 \otimes U \subset U_{\mathbf t\eta^{\vee}}$. Similarly, we can identify $I(\theta(J),\phi(U^{r,\eta^{\vee}}))$ with $T^0(\phi(U))/T^r(\phi(U))$. 

Since $\Delta^U_{\mathbf t\eta^{\vee}}(T^r(U)) \subset T^r(\phi(U))$, the map $\Delta_{\mathbf t\eta^{\vee}}$ induces a map, denoted 
\[ \widetilde{\Delta}^{U^{r,\eta^{\vee}}}: T^0(U)/T^r(U) \rightarrow T^0(\phi(U))/T^r(\phi(U)) .\]

\begin{lemma}\label{lem image of delta}
$ \mathrm{im} \widetilde{\Delta}^{U^{r,\eta^{\vee}}} \cong \mathrm{im} \Delta^{U^{r,\eta^{\vee}}} $.
\end{lemma}

\begin{proof}
By using Frobenius reciprocity, Lemma \ref{lem bullet dual ismo}, Lemma \ref{lem bullet dual} and Lemma \ref{lem decompose for lang class}, we have 
\[ \mathrm{Hom}_{\mathbb{H}}(I(J,U), I(\theta(J), \phi(U))) \cong \mathrm{Hom}_{\mathbb{H}_J}(U^{r,\eta^{\vee}}, U^{r,\eta^{\vee}}) \cong \mathbb{C}^r. \]
Furthermore, for an element $\psi \in \mathrm{Hom}_{\mathbb{H}}(I(J,U^{r,\eta^{\vee}}), I(\theta(J), \phi(U^{r,\eta^{\vee}}))) $, $\mathrm{im} \psi$ is determined by $\mathrm{im} \psi|_{1\otimes U^{r,\eta^{\vee}}}$. On the other hand by counting dimensions, one can conclude that 
\[ \widetilde{\Delta}^{U^{r,\eta^{\vee}}}(1 \otimes U^{r,\eta^{\vee}}) \cong \Delta^{U^{r,\eta^{\vee}}}(1 \otimes U^{r,\eta^{\vee}}) \cong 1 \otimes U^{r,\eta^{\vee}} .\]
Combining all these, we prove the lemma. 
\end{proof}

\subsection{A quotient of generalized standard modules}

\begin{lemma} \label{lem unique sub}
Let $\mathbb{A}$ be a complex associative algebra with an unit. Let $X$ be a finite-dimensional $\mathbb{A}$-module. Fix a finite collection $\left\{ L_1, \ldots, L_r \right\}$ of simple $\mathbb{A}$-modules. There exists a unique submodule $N$ such that
\begin{enumerate}
\item the composition factors of the socle of $X/N$ are isomorphic to some $L_i$, 
\item no composition factors of $N$ are isomorphic to some $L_i$.
\end{enumerate}
\end{lemma}

\begin{proof}
Let $\mathcal M$ be the set of all proper submodules of $X$ whose composition factors are not isomorphic to any of $L_i$. $\mathcal M$ is nonempty since the zero module is in $\mathcal M$. Suppose $N_1$ and $N_2$ are two maximal element (with respect to the inclusion) in $\mathcal M$. By considering $N_1+N_2$ which has no composition factor of $L_i$ and using the maximality of $N_1$ and $N_2$, we have $N_1=N_2$. (To see $N_1+N_2$ has no composition factor of $L_i$, we can use $(N_1+N_2)/N_2 \cong N_1/(N_1 \cap N_2)$ and apply the Jordan-H\"older Theorem for composition factors by finite dimensionality.) Hence $\mathcal N$ has a unique maximal element. Let $N$ be the maximal element in $\mathcal N$. Then $N$ automatically satisfies (2). To show $N$ also satisfies the property (1), suppose there exists a simple module $L$ such that $\mathrm{Hom}_{\mathbb{A}}(L, X/N)\neq 0$ and $L$ is non-isomorphic to any of $L_i$. Let $0 \neq f \in \mathrm{Hom}_{\mathbb{A}}(L, X/N)\neq 0$. Then $f(L)+N \in \mathcal N$, contradicting the maximality of $N$. Hence $N$ also satisfies (1) and this proves the existence part of the lemma. The uniqueness part follows from the fact that a module satisfying (1) and (2) will lie in $\mathcal M$.
\end{proof}


We now describe quotients of some generalized standard modules of strict $S(V)$-type in terms of the Jantzen filtration of (ordinary) standard modules. In the case that the generalized standard module is the standard module, the quotient simply gives the unique simple quotient in the Langlands classification. 
\begin{theorem} \label{thm JF general stand}
Let $(J,U) \in \Xi_L$ (Definition \ref{def generalized St}) and let $0 \neq \eta^{\vee} \in V_J^{\bot, \vee}$. Let $U^{r-1,\eta^{\vee}}$ and $U^{r, \eta^{\vee}}$ be as in Section \ref{ss realization sv ext}. Let $\iota: I(J,U^{r-1, \eta^{\vee}}) \hookrightarrow I(J, U^{r, \eta^{\vee}}) \cong \mathrm{JF}_{\eta^{\vee}}^0(J,U^{r, \eta^{\vee}})$ be the inclusion map (unique up to a scalar).   Then
\begin{enumerate}
\item The map $\iota$ induces an injective map from $L(J, U^{r-1,\eta^{\vee}})$ to $L(J,U^{r,\eta^{\vee}})$.
\item $\mathrm{JF}^{0}_{\eta^{\vee}}(J, U)/\mathrm{JF}^{r}_{\eta^{\vee}}(J, U)$ is isomorphic to $  L(J,U^{r,\eta^{\vee}})/L(J,U^{r-1,\eta^{\vee}}) $.
\item $\mathrm{JF}^0_{\eta^{\vee}}(J, U^{r, \eta^{\vee}})/\mathrm{JF}^1_{\eta^{\vee}}(J, U^{r, \eta^{\vee}}) \cong L(J,U^{r, \eta^{\vee}})$ is the unique indecomposable quotient of $I(J,U^{r, \eta^{\vee}})$ with the properties that: (a) the module has unique simple quotient and unique simple submodule, both of which are isomorphic to $L(J,U)$ and (b) the multiplicity of $L(J,U)$ in the composition series of the module is the same as that in the composition series of $I(J,U^{r, \eta^{\vee}})$.
\item $L(J,U^{r, \eta^{\vee}})$ is self $\bullet$-dual.
\end{enumerate}
\end{theorem}

\begin{proof}

For simplicity, we shall write $\widetilde{\Delta}_i$ for $\widetilde{\Delta}^{U^{i,\eta^{\vee}}}$.

We first prove (1). By Lemma \ref{lem image of delta}, we have isomorphisms $\mathrm{coker} \widetilde{\Delta}^{U^{i,\eta^{\vee}}} \cong \mathrm{coker} \Delta^{U^{i,\eta^{\vee}}} \cong L(J,U^{i,\eta^{\vee}})$ ($i=r-1,r$). The map $\iota$ induces a natural map from $\iota': I(J,U^{r-1,\eta^{\vee}}) \cong T^{0}(U)/T^{r-1}(U) \hookrightarrow T^0(U)/T^r(U) \cong I(J,U^{r,\eta^{\vee}})$ given by a multiplication of $\mathbf t$. Hence we have a map $\widetilde{\iota}: T^0(U)/T^{r-1}(U) \rightarrow \mathrm{coker} \widetilde{\Delta}_{r}$. By comparing the maps $\widetilde{\Delta}_{r-1}, \widetilde{\Delta}_r$, it is straightforward to verify that $\mathrm{ker} \iota =\mathrm{ker} \widetilde{\Delta}_{r-1}$. Hence we obtain an induced injective map from $\mathrm{coker} \Delta^{U^{r-1,\eta^{\vee}}}$ to  $\mathrm{coker} \Delta^{U^{r,\eta^{\vee}}}$.

We now prove (2). 
\[\xymatrix{ 
0 \ar[r] & T^0(U)/T^{r-1}(U) \ar[r]^{\iota'} &     T^0(U)/T^r(U)            \ar[d]^{\widetilde{\Delta}_r}    \ar[r]^{F}                   &      T^0(U)/T^1(U)    \ar[d]^{\overline{\Delta}^r}  \ar[r] & 0 \\
  & &   T^0(\phi(U))/T^r(\phi(U))  & T^0(\phi(U))/(T^r(\phi(U))+\Delta_{\mathbf t\eta^{\vee}}(T^1(U))) \\
   }
\]
where $F$ is the natural surjective map so that the top sequence is exact. Define $F'=\widetilde{\Delta}^1 \circ F$. We claim that $\mathrm{ker} F'=\mathrm{im} \iota' + \mathrm{ker} \widetilde{\Delta}^r$. From this, we have 
\[ \frac{(T^0(U)/T^r(U)) }{\mathrm{ker}F'} \cong \frac{(T^0(U)/T^r(U) )}{(\mathrm{im} \iota' + \mathrm{ker} \widetilde{\Delta}_r) }\cong \frac{\mathrm{coker} \widetilde{\Delta}_r}{(\mathrm{im} \iota' + \mathrm{ker} \widetilde{\Delta}_r)/ \mathrm{ker} \widetilde{\Delta}_r } \cong \frac{L(J,U^{r,\eta^{\vee}})}{L(J,U^{r-1,\eta^{\vee}})}.
\]
The last isomorphism follows from (1). On the other hand, we have
\[ \frac{(T^0(U)/T^r(U)) }{\mathrm{ker}F'} \cong \frac{T^0(U)/T^1(U)}{\mathrm{ker} \overline{\Delta}^r} \cong \frac{T^0(U)}{\mathrm{ker}\Delta_{\mathbf t\eta^{\vee}}^{r,U}+T^1(U)} \cong \frac{\mathrm{JF}^0_{\eta^{\vee}}(J,U)}{\mathrm{JF}^r_{\eta^{\vee}}(J,U)} .\]
This shall prove (2). The first isomorphism above follows from the fact that $F$ is surjective. For the details of the second isomorphism, one can deduce from the discussions in Section \ref{ss line inde lemma}. (More precisely, using Proposition \ref{prop strong linearly indpende}, one can show that $\mathrm{ker} \overline{\Delta}^r \cong (\mathrm{ker}\Delta_{\mathbf t\eta^{\vee}}^{r,U}+T^1(U))/T^1(U)$.)

 We now turn to prove the claim $\mathrm{ker} F'=\mathrm{im} \iota' + \mathrm{ker} \widetilde{\Delta}_r$. The inclusions $\mathrm{im} \iota' \subset \mathrm{ker} F'$ and $\mathrm{ker} \widetilde{\Delta}_r \subset \mathrm{ker} F'$ follow from definitions. We now pick a representative $\widetilde{x}_{\mathbf t} \in T^0(U)$ of an element in $\mathrm{ker} F'$. Write $\widetilde{x}_{\mathbf t}$ as the form:
\[ \sum_{i=0}^{r-1} \mathbf t^i\sum_{p\in \mathbb{Z}_{\geq 0}\cup \left\{ \infty \right\}} x^{p,i}_{\mathbf t}+\mathbf t^r z_{\mathbf t},
\]
where $x^{p,i}_{\mathbf t} \in X^p_{\mathbf t}$ and $z_{\mathbf t} \in T^0(U)$. Here $X^p_{\mathbf t}$ is defined as in Section \ref{ss line inde lemma}. For simplicity, set $y^{p,i}_{\mathbf t}=\frac{1}{\mathbf t^p}\Delta_{\mathbf t\eta^{\vee}}(x^{p,i}_{\mathbf t})$. Let $N$ be the least integer such that there exists a pair $(p,i)$ with $p+i=N$ and $x^{p,i} \neq 0$. If such integer exists and $N \leq r-1$, then the element $\Delta^{U^{r,\eta^{\vee}}}_{\mathbf t\eta^{\vee}}(\widetilde{x}_{\mathbf t})$ can be written as
\[   \mathbf t^N( y_{\mathbf t}^{0,N}+y_{\mathbf t}^{1,N-1}+\ldots, y_{\mathbf t}^{N,0}) +\mathbf t z_{\mathbf t}' ,
\]
where $z'_{\mathbf t} \in T^0(\phi(U))$. Recall that $\widetilde{x}_{\mathbf t}$ is a representative of an element in $\mathrm{ker} \widetilde{\Delta}^r$. We must then have $y_{\mathbf t}^{0,N}=0$ by Proposition \ref{prop strong linearly indpende} and so $x_{\mathbf t}^{0,N}=0$. Let $x'_{\mathbf t}=\mathbf t x_{\mathbf t}^{1,N}+\ldots +\mathbf t^{N}x_{\mathbf t}^{N,0} \in T^1(U)$. Now we consider the element $x_{\mathbf t}-x_{\mathbf t}'$. Repeating the above process we shall eventually obtain an element $x''_{\mathbf t} \in T^1(U)$ such that $\Delta_{\mathbf t\eta^{\vee}}^{U^{r,\eta^{\vee}}}(x_{\mathbf t}-x''_{\mathbf t}) \in T^r(\phi(U))$ and $x''_{\mathbf t} \in T^1(U)$. By definition $x''_{\mathbf t}+T^r(U) \in \mathrm{im} \iota'$ and hence $x_{\mathbf t}+T^r(U) \in \mathrm{\im} \iota'+\mathrm{ker} F'$ as desired.

We now prove (3). Let $Z=\mathrm{JF}^0_{\eta^{\vee}}(J,U^{r,\eta^{\vee}})/\mathrm{JF}^1_{\eta^{\vee}}(J,U^{r,\eta^{\vee}})$. Since $\mathrm{JF}^0(J, U^{r,\eta^{\vee}})/\mathrm{JF}^1(J,U^{r,\eta^{\vee}})$ is a quotient of $I(J,U^{r,\eta^{\vee}})$, we have a surjective map from $\mathbb{H} \otimes_{\mathbb{H}_J} U^{r,\eta^{\vee}}$ to $\mathrm{JF}^0(J,U^{r,\eta^{\vee}})$. By the left-exactness of $\mathrm{Hom}_{\mathbb{H}}(., Y)$, we have an injection from $\mathrm{Hom}_{\mathbb{H}}(Z,Y)$ to $\mathrm{Hom}_{\mathbb{H}}(I(J,U^{r,\eta^{\vee}}),Y)$. Then $Z$ having a unique simple quotient isomorphic to $L(J,U)$ follows from the fact that $I(J,U)$ has a unique simple quotient isomorphic to $L(J,U)$. This also implies $Z$ is indecomposable. By definitions,
\[  T^0(U^{r,\eta^{\vee}})/\mathrm{ker} \Delta_{\mathbf t\eta^{\vee}}^{U^{r,\eta^{\vee}}} \cong \mathrm{im} \Delta_{\mathbf t\eta^{\vee}}^{U^{r,\eta^{\vee}}}/ T^1(\phi(U^{r,\eta^{\vee}})).
\]
Hence $Z$ is also isomorphic to a submodule of $T^0(\phi(U^{r,\eta^{\vee}}))/T^1(\phi(U^{r,\eta^{\vee}})) \cong I(\theta(J), \phi(U))$. Now by Lemma \ref{lem decompose for lang class} and Frobenius reciprocity, $I(J, U^{r, \eta^{\vee}})$ has a unique simple quotient and so does $L(J, U^{r, \eta^{\vee}})$. By considering $I(\theta(J), \phi(U))^*$ as in the proof of Proposition \ref{prop intertwin properties}, $I(\theta(J), \phi(U^{r, \eta^{\vee}}))^*$ has a unique simple quotient and so $I(\theta(J), \phi(U^{r, \eta^{\vee}}))$ has a unique simple submodule. Hence $L(J, U^{r, \eta^{\vee}})$ also has a unique simple submodule. This proves the property (a). Property (b) follows from an induction argument using (2), which we have proved. 

The uniqueness for (3)  follows from Lemma \ref{lem unique sub}.

We now consider (4). For notation simplicity, set $G=\Delta^{U^{r, \eta^{\vee}}}$. Then we have a dual map $G^{\bullet}: I(\theta(J), \phi(U^{r, \eta^{\vee}}))^{\bullet} \rightarrow I(J, U^{r, \eta^{\vee}})^{\bullet}$. Then the image and the cokernel of $G^{\bullet}$ is isomorphic to $L(J, U^{r, \eta^{\vee}})^{\bullet}$. On the other hand, by Lemma \ref{lem bullet dual}, we have 
\[  I(\theta(J), \phi(U^{r, \eta^{\vee}}))^{\bullet} \cong I(J, U^{r, \eta^{\vee}}), \quad I(J, U^{r, \eta^{\vee}})^{\bullet} \cong I(\theta(J),  \phi(U^{r, \eta^{\vee}})) .
\]
By using (2) and Lemma \ref{lem bullet dual ismo}, we have $\mathrm{im}\ G^{\bullet} \cong \mathrm{im}\ \Delta^{U^{r, \eta^{\vee}}}$. Hence
\[ L(J, U^{r, \eta^{\vee}}) \cong L(J, U^{r, \eta^{\vee}})^{\bullet} . \qedhere
\]
\end{proof}

\begin{remark}
Generalizing Theorem \ref{thm JF general stand} to other layers of Jantzen filtrations of $I(J,U^{r,\eta^{\vee}})$ seems to be harder or less direct from our approach (because the formula involves derivatives and it is hard to apply). We give an example to illustrate the Jantzen filtration of a higher layer may be more complicated. Let $\mathbb{H}$ be of type $A_1$ and let $k=1$. Let $\alpha$ be the unique simple root of type $A_1$. Let $U$ be the unique irreducible $S(V)$-module with the weight $\frac{1}{2}\alpha^{\vee}$. We consider the Jantzen filtration of $\mathbb{H} \otimes_{S(V)} U^{2,\eta^{\vee}}$, where $\eta^{\vee}=\frac{1}{2}\alpha^{\vee}$. Take $v_0=\alpha$. Let
\[ \widetilde{x}_{\mathbf t}= \left( t_{s_{\alpha}}-\frac{1}{\alpha} \right) \otimes\left(0,   u \right) .\]
We have
\begin{align}
\Delta_{\mathbf t\eta^{\vee}}(\widetilde{x}_{\mathbf t})  &= 1 \otimes \left( 0, \frac{\mathbf t^2+2\mathbf t}{(\mathbf t+1)^2} u \right) \in T^1(\phi(U_{\mathbf t\eta^{\vee}})) \setminus T^2(\phi(U_{\mathbf t\eta^{\vee}}))
\end{align}
It might be tempted to think in the beginning that 
$x_{\mathbf t}|_{\mathbf t=0}\in \mathrm{JF}^1(J, U^{2,\eta^{\vee}}) \setminus \mathrm{JF}^2(\emptyset,U^{2,\eta^{\vee}})$, which however is false.
Let 
\[ y_{\mathbf t} = -\left( t_{s_{\alpha}}-\frac{1}{\alpha} \right) \otimes \left( u, 0 \right) .\]
Then $\Delta_{\mathbf t\eta^{\vee}}(\mathbf ty_{\mathbf t}+ x_{\mathbf t}) \in T^2(\phi(U))$. Indeed we have $ \mathrm{JF}^1(\emptyset,U^{2,\eta^{\vee}})=\mathrm{JF}^2(\emptyset,U^{2,\eta^{\vee}})$ and $\mathrm{JF}^3(\emptyset, U^{2,\eta^{\vee}})=0$. 
\end{remark}


\subsection{Good and bad directions }

We now define a set, which will be used to parametrize certain self-extensions of simple modules.

\begin{definition} \label{def bad dir}
For $\eta^{\vee} \in V_J^{\bot, \vee}$, we say $\eta^{\vee}$ is in a {\it bad direction} (with respect to $J$ and $U$) if
\[  \mathrm{JF}^1_{\eta^{\vee}}(J,U)=\mathrm{JF}^2_{\eta^{\vee}}(J,U) .
\]
Otherwise we say $\eta^{\vee}$ is a {\it good direction}. Denote by $V^{\bot}_{\mathrm{bad}}(J,U)$, or simply $V^{\bot}_{\mathrm{bad}}$, the set of vectors in a bad direction. 
\end{definition}

\begin{example}\label{ex A2 bad good}
We consider $\mathbb{H}$ to be of type $A_2$ with $k \equiv 1$. Let $\alpha, \beta$ be the simple roots. Consider the central character $\gamma^{\vee}=\frac{1}{2}\alpha^{\vee} +C(\beta^{\vee}+2\alpha^{\vee}) $, where $C$ is taken to be a sufficiently large positive number. In particular, $\gamma^{\vee}$ is in the dominant chamber. We consider the standard module $X=\mathbb{H} \otimes_{S(V)} \mathbb{C}_{\gamma^{\vee}}$. By a simple computation, we have that $X$ contains two composition factors. 

The element $\frac{1}{2}(\beta^{\vee}+2\alpha^{\vee})$ is in a bad direction. We have 
\[ \mathrm{JF}^i_{\frac{1}{2}(\beta^{\vee}+2\alpha^{\vee})}(\emptyset, \mathbb{C}_{\gamma^{\vee}}) \cong L(\emptyset, \mathbb{C}_{\gamma^{\vee}})\]
 for all $i \geq 1$. In contrast, the element $\frac{1}{2}\alpha^{\vee}$ is a good direction since $\mathrm{JF}_{\frac{1}{2}\alpha^{\vee}}^1(\emptyset, \mathbb{C}_{\gamma^{\vee}})$ is the unique simple submodule of $X$ and $\mathrm{JF}_{\frac{1}{2}\alpha^{\vee}}^2(\emptyset, \mathbb{C}_{\gamma^{\vee}})=0$.
\end{example}

\begin{remark}
When the central character supports a tempered module, we expect that those corresponding standard modules have ''less'' bad directions (or even no bad directions). In contrast, for the principal series of a generic central character, the whole space $V^{\vee}$ is the set of bad directions. Example \ref{ex A2 bad good} illustrates an example in between of the previous two cases. We expect that the occurrence of bad directions depends on how ''generic'' the central character of the standard module is. 

\end{remark}

\begin{proposition}
The set $V^{\bot}_{\mathrm{bad}}(J,U)$ forms a vector space i.e. closed under addition and scalar multiplication. 
\end{proposition}

\begin{proof}
We can naturally identify $T^i(\phi(U_{\mathbf t\eta^{\vee}_1}))$, $T^1(\phi(U_{\mathbf t\eta^{\vee}_2}))$ and $T^i(\phi(U_{\mathbf t(\eta^{\vee}_1+\eta_2^{\vee})}))$ as vector spaces (via the grading by $\mathbf t$) and simply write $T^i(\phi(U))$. We also similarly do for $T^0(U)$. Let $x_{\mathbf t} \in T^0(U)$. If $\Delta_{\mathbf t\eta^{\vee}_1}(x_{\mathbf t}), \Delta_{\mathbf t\eta^{\vee}_2}(x_{\mathbf t}) \in T^1(\phi(U))$, then direct calculation from definitions gives that
\[\Delta_{\mathbf t\eta^{\vee}}(x_{\mathbf t})+ \Delta_{\mathbf t\eta^{\vee}_1}(x_{\mathbf t})-\Delta_{\mathbf t(\eta^{\vee}_1+\eta_2^{\vee})}(x_{\mathbf t}) \in T^2(\phi(U)) . 
\]
Note that we also have $0 \in V^{\bot}_{\mathrm{bad}}(J,U)$.
\end{proof}

\subsection{First self-extension}
Theorem \ref{thm JF general stand} provides structural information for a quotient of generalized standard modules. We shall show in Theorem \ref{thm jantan differential} how to obtain some information for $\mathrm{Ext}$-group from those information. One may expect to obtain some other information from those quotients (see Example \ref{ex yoneda product}).



Consider the short exact sequence
\[   0\rightarrow N(J,U) \rightarrow I(J,U) \stackrel{\mathrm{pr} }{\rightarrow} L(J, U) \rightarrow 0.
\]

\begin{lemma} \label{lem identify induced simple}
Let $(J,U) \in \Xi_L$. Then $\mathrm{pr}$ induces an isomorphism $\mathrm{Ext}^i_{\mathbb{H}}(I(J,U), I(J,U)) \cong \mathrm{Ext}^i_{\mathbb{H}}(I(J,U), L(J,U))$. 
\end{lemma}

\begin{proof}
We apply the functor $\mathrm{Hom}_{\mathbb{H}}(I(J,U),.)$ to the short exact sequence before the lemma to obtain a long exact sequence. Then by Proposition \ref{prop ext lower langlands para} and Lemma \ref{lem decompose for lang class}, we have $\mathrm{Ext}^i_{\mathbb{H}}(I(J,U), N(J,U))=0$. Thus $\mathrm{Ext}^i_{\mathbb{H}}(I(J,U), I(J,U)) \cong \mathrm{Ext}^i_{\mathbb{H}}(I(J,U), L(J,U))$ via the induced map.
\end{proof}

\begin{theorem} \label{thm jantan differential}
Let $(J,U) \in \Xi_L$ (Definition \ref{def generalized St}). Recall that $V_{\mathrm{bad}}^{\bot}(J,U)$ is defined in Definition \ref{def bad dir}. Let
\[ \mathrm{pr}^{*,i}: \mathrm{Ext}^i_{\mathbb{H}}(L(J,U), L(J,U)) \rightarrow  \mathrm{Ext}^i_{\mathbb{H}}(I(J,U), L(J,U)) \]
be the natural map induced from the surjective map $I(J,U) \rightarrow L(J,U)$. Then 
\begin{enumerate}
\item Identify $\mathrm{Ext}^1_{\mathbb{H}}(I(J,U), I(J,U)) \cong \mathrm{Ext}^1_{\mathbb{H}}(I(J,U), L(J,U))$ via Lemma \ref{lem identify induced simple}.
Identify \begin{align*} \label{eqn ext standard same2}
  \mathrm{Ext}^1_{\mathbb{H}}(I(J,U), I(J,U)) &\cong \mathrm{Ext}^1_{S(V_J^{\bot})}(L, L) \oplus \mathrm{Ext}^1_{\mathbb{H}_J^{ss}}(\overline{U}, \overline{U}) 
\end{align*}
as in Lemma \ref{lem nat isom ext} and identify $\mathrm{Ext}^1_{S(V_J^{\bot})}(L, L)\cong V_J^{\vee, \bot}$. Then 
\[\mathrm{im}\ \mathrm{pr}^{*,1} \cap V_J^{\vee, \bot} \cong V_{\mathrm{bad}}^{\bot}(J,U) .\]
\item Suppose $\overline{U}$ is a discrete series or more generally $ \mathrm{Ext}^1_{\mathbb{H}_J^{ss}}(\overline{U}, \overline{U})=0$. Then 
\[\mathrm{Ext}^1_{\mathbb{H}}(L(J,U), L(J,U)) \cong V_{\mathrm{bad}}^{\bot}(J,U). \]
\end{enumerate}
\end{theorem}

 \begin{proof}
We first consider (1). Consider the following short exact sequence:
\[    0 \rightarrow N(J,U) \rightarrow I(J, U) \rightarrow L(J,U) \rightarrow 0.
\]
This induces a long exact sequence of the following form:
\[ \ldots \rightarrow \mathrm{Hom}_{\mathbb{H}}(N(J,U), L(J,U)) \rightarrow \mathrm{Ext}^1_{\mathbb{H}}(L(J,U), L(J,U)) \stackrel{\mathrm{pr}^{*,1}}{\rightarrow} \mathrm{Ext}^1_{\mathbb{H}}(I(J,U), L(J,U)) \rightarrow \ldots
\]

Suppose $ \mathrm{im}\ \mathrm{pr}^{*,1}\cap V_J^{\vee, \bot} \setminus V_{\mathrm{bad}}^{\bot} \neq \emptyset$. Let $\eta^{\vee} \in  \mathrm{im}\ \mathrm{pr}^{*,1}\cap V_J^{\vee, \bot} \setminus V_{\mathrm{bad}}^{\bot}(J,U)$ and let $\eta' \in \mathrm{Ext}^1_{\mathbb{H}}(L(J,U), L(J,U))$ such that $\mathrm{pr}^{*,1}(\eta')=\eta^{\vee}$. Let $\mathcal E(\eta')$ be the $\mathbb{H}$-modules constructed from the Yoneda first extension (see e.g. \cite[Theorem 3.4.3]{We}) for $\eta'$ respectively. Recall that we are working for several identification. We now regard $\eta^{\vee}$ as an element in $\mathrm{Ext}^1_{\mathbb{H}}(I(J,U), I(J,U))$ and let $\mathcal E(\eta^{\vee})$ be the $\mathbb{H}$-module constructed from the Yoneda first extension for $\eta^{\vee}$.

From the construction of the Yoneda extension and tracing identifications, the map $I(J,U) \rightarrow L(J,U)$ being surjective implies that the induced map $\mathcal E(\eta^{\vee}) \rightarrow \mathcal E(\eta')$ is also surjective. This implies that there exists a subquotient of $\mathcal E(\eta^{\vee})$, in which all the compositions factors are isomorphic to $L(J,U)$. By the uniqueness statement in Theorem \ref{thm JF general stand}(2), $\mathrm{JF}^0_{\eta^{\vee}}(J, U^{2,\eta^{\vee}})/ \mathrm{JF}^1_{\eta^{\vee}}(J, U^{2,\eta^{\vee}})\cong L(J, U^{2, \eta^{\vee}})$ is isomorphic to $\mathcal E(\eta')$. 

On the other hand, by Theorem \ref{thm JF general stand}(1), the composition factors of $\mathrm{JF}^0_{\eta^{\vee}}(J, U^{2,\eta^{\vee}})/ \mathrm{JF}^1_{\eta^{\vee}}(J, U^{2,\eta^{\vee}})$ contains composition factors of $\mathrm{JF}^0_{\eta^{\vee}}(J, U)/ \mathrm{JF}^2_{\eta^{\vee}}(J, U)$. However, since $\mathrm{JF}^1_{\eta^{\vee}}(J, U) \neq \mathrm{JF}^2_{\eta^{\vee}}(J, U)$,  $\mathrm{JF}^0_{\eta^{\vee}}(J, U)/ \mathrm{JF}^2_{\eta^{\vee}}(J, U)$ contains a composition factor other than $L(J,U)$ and so does $L(J, U^{2, \eta^{\vee}})$. This gives a contradiction to the above conclusion that $L(J, U^{2, \eta^{\vee}}) \cong \mathcal E(\eta')$. This proves  $\mathrm{im}\ \mathrm{pr}^{*,1} \cap V^{\bot, \vee}_J \subset V_{\mathrm{bad}}^{\bot}(J,U)$.

For the converse inclusion, let $0 \neq \eta^{\vee} \in V_{\mathrm{bad}}^{\bot}$. Then 
\[ \mathrm{JF}_{\eta^{\vee}}^0(J,U)/\mathrm{JF}_{\eta^{\vee}}^2(J,U)=\mathrm{JF}_{\eta^{\vee}}^0(J,U)/\mathrm{JF}_{\eta^{\vee}}^1(J,U) \cong L(J,U) \]
 by the definition of bad directions. Hence, by Theorem \ref{thm JF general stand}(1),
\[ \mathrm{JF}^0_{\eta^{\vee}}(J, U^{2, \eta^{\vee}})/(\mathrm{JF}^{1}_{\eta^{\vee}}(J, U^{2,\eta^{\vee}})  + \mathrm{im} \iota) \cong L(J,U) ,\]
where $\iota$ is the natural embedding from $I(J,U)$ to $I(J,U^{2,\eta^{\vee}})$. We also have $\im \iota/(\im \iota \cap \mathrm{JF}^{1}_{\eta^{\vee}}(J, U^{2,\eta^{\vee}})) \cong L(J,U)$ by the definition of intertwining operators and the definition of $L(J,U)$ (see Theorem \ref{thm JF general stand}(1)). For simplicity, let $E=\mathrm{JF}^0_{\eta^{\vee}}(J, U^{2, \eta^{\vee}})/\mathrm{JF}^{1}_{\eta^{\vee}}(J, U^{2,\eta^{\vee}})$. The above facts imply that $E$ is an indecomposable module of length $2$ and all the composition factors isomorphic to $L(J,U)$. Since $I(J,U^{2,\eta^{\vee}}) \cong \mathrm{JF}^0_{\eta^{\vee}}(J,U^{2,\eta^{\vee}})$, we have a surjection from $I(J,U^{2,\eta^{\vee}})$ to $E$. Since $E$ has a unique submodule isomorphic to $L(J,U)$, the surjection map induces a commutative diagram:
\[\xymatrix{ 
                  0         \ar[r] & I(J,U)                  \ar[d]^{\mathrm{pr}'}    \ar[r]                       &          I(J,U^{2,\eta^{\vee}})  \ar[d]  \ar[r] & I(J,U) \ar[d]^{\mathrm{pr}} \ar[r] & 0                 \\
    0         \ar[r] & L(J,U)                       \ar[r]                   &          E   \ar[r] & L(J,U)  \ar[r] & 0                 \\ }
\]
Here $\mathrm{pr}'$ is a non-zero scalar multiple of $\mathrm{pr}$. Then we have a natural commutative diagram of the following form:
\[\xymatrix{ 
                  0         \ar[r] & I(J,U)/N(J,U)                  \ar[d]    \ar[r]                       &          I(J,U^{2,\eta^{\vee}})/N(J,U)  \ar[d]  \ar[r] & I(J,U) \ar[d]^{\mathrm{pr}}  \ar[r] & 0                 \\
    0         \ar[r] & L(J,U)                       \ar[r]                   &          E   \ar[r] & L(J,U)  \ar[r] & 0                 \\ }
\]
Recall that $I(J,U)/N(J,U) \cong L(J,U)$. Denote by $\eta' \in \mathrm{Ext}^1_{\mathbb{H}}(L(J,U),L(J,U))$ for the corresponding element of the bottom short exact sequence (under the Yoneda correspondence). It follows from definitions that $\mathrm{pr}^{*,1}(\eta')=\eta^{\vee}$. This completes the proof for (1).

For (2), we have the long exact sequence:
\[ \ldots \rightarrow \mathrm{Hom}_{\mathbb{H}}(N(J,U), L(J,U)) \rightarrow \mathrm{Ext}^1_{\mathbb{H}}(L(J,U), L(J,U)) \stackrel{\mathrm{pr}^{*,1}}{\rightarrow} \mathrm{Ext}^1_{\mathbb{H}}(I(J,U), L(J,U)) \rightarrow \ldots \]
Since $\mathrm{Hom}_{\mathbb{H}}(N(J,U), L(J,U))=0$, $\mathrm{pr^{*,1}}$ is injective and hence $\mathrm{Ext}^1_{\mathbb{H}}(L(J, U), L(J, U)) \cong \mathrm{im}\ \mathrm{pr}^{*,1}$. Combining with the result of (1) and assumptions, we obtain (2). 
\end{proof}

\section{First extensions and filtrations} \label{s first ext}

\subsection{First extensions}
We now summarize our study and state our main result concerning $\mathrm{Ext}^1_{\mathbb{H}}$ for some simple modules. 

\begin{theorem} \label{thm first ext sum}
Let $\mathbb{H}$ be the graded Hecke algebra as in Definition \ref{def graded affine}. Let $(J_1,U_1), (J_2, U_2) \in \Xi_L$ (Definition \ref{def generalized St}). Set $X=L(J_1,U_1)$ and let $Y=L(J_2,U_2)$ (see Definition \ref{def generalized St} for notations).  Then
\begin{enumerate}
\item If $\nu(Y) < \nu(X)$, then $\mathrm{Ext}^1_{\mathbb{H}}(X,Y)\cong \mathrm{Hom}_{\mathbb{H}}(N(J_1,U_1),Y)$.
\item If $\nu(X) < \nu(Y)$, then $\mathrm{Ext}^1_{\mathbb{H}}(X,Y) \cong \mathrm{Ext}^1_{\mathbb{H}}(Y,X) \cong \mathrm{Hom}_{\mathbb{H}}(N(J_2,U_2),X)$.
\item If $\nu(Y)$ and $\nu(X)$ are incomparable, then $\mathrm{Ext}^1_{\mathbb{H}}(X,Y)=0$.
\item Suppose $\nu(X)=\nu(Y)$  (and in particular $J_1=J_2$). Further suppose that 
\[\mathrm{Ext}^1_{\mathbb{H}_{J_1}^{ss}}(\mathrm{Res}_{\mathbb{H}_{J_1}^{ss}}U_1, \mathrm{Res}_{\mathbb{H}_{J_1}^{ss}}U_2)=0 .\]
\begin{enumerate}
\item If $U_1 \not\cong U_2$, then $\mathrm{Ext}^1_{\mathbb{H}}(X, Y)=0$. 
\item If $U_1 \cong U_2$, then $\mathrm{Ext}^1_{\mathbb{H}}(X, Y) \cong V_{\mathrm{bad}}^{\bot}(J_1, U_1)$. 
\end{enumerate}
\end{enumerate}
\end{theorem}

\begin{proof}
 We have the following short exact sequence:
\[  0 \rightarrow N(J_1,U_1) \rightarrow I(J_1,U_1) \rightarrow L(J_1,U_1) \rightarrow 0 .\]
By applying the $\mathrm{Hom}_{\mathbb{H}}( , Y)$ functor, we have
\[ \ldots \rightarrow \mathrm{Hom}_{\mathbb{H}}(I(J_1,U_1), Y)  \rightarrow \mathrm{Hom}_{\mathbb{H}}(N(J_1,U_1), Y) \rightarrow \mathrm{Ext}_{\mathbb{H}}^1(L(J_1, U_1), Y) \rightarrow \mathrm{Ext}^1_{\mathbb{H}}(I(J_1,U_1), Y) \rightarrow \ldots
\]
We first consider (1) and (3) and so suppose $\nu(X) \not\leq \nu(Y)$. Then using Proposition \ref{prop ext lower langlands para}, we have $\mathrm{Ext}_{\mathbb{H}}^1(L(J_1, U_1), Y) \cong \mathrm{Hom}_{\mathbb{H}}(N(J_1, U_1), Y)$. This proves (1). For $\nu(Y)$ and $\nu(X)$ being incomparable, we also have $\mathrm{Hom}_{\mathbb{H}}(N(J_1, U_1), Y)=0$ by Lemma \ref{lem decompose for lang class}. This proves (3).

For (2), there is a natural isomorphism between $\mathrm{Ext}^1_{\mathbb{H}}(X,Y) \cong \mathrm{Ext}^1_{\mathbb{H}}(Y^{\bullet},X^{\bullet})$. (2) then follows from (1) and Lemma \ref{lem bullet dual ismo}.


We now consider (4).  Write $U_i =\overline{U}_i \otimes L_i$ as $\mathbb{H}_J \cong \mathbb{H}_J^{ss} \otimes S(V_J^{\bot})$-algebras, where $\overline{U}_i$ is $\mathbb{H}^{ss}_J$-tempered module and $L_i$ is a one-dimensional $S(V_J^{\bot})$-module. By Lemma \ref{lem nat isom ext}, we have
\[\mathrm{Ext}^1_{\mathbb{H}}(I(J_1, U_1), L(J_2,U_2)) \cong \mathrm{Ext}^1_{\mathbb{H}^{ss}_J}(\overline{U}_1,  \overline{U}_2) \otimes \mathrm{Hom}_{S(V_J^{\bot})}(L_1, L_2) \oplus\mathrm{Hom}_{S(V_J^{\bot})}(\overline{U}_1, \overline{U}_2) \otimes  \mathrm{Ext}^1_{S(V_J^{\bot})}(L_1, L_2) .\]
For (4)(a), by using $\mathrm{Hom}_{\mathbb{H}_J^{ss}}(\overline{U}_1, \overline{U}_2)=0$ and $\mathrm{Ext}^1_{\mathbb{H}^{ss}_J}(\overline{U}_1, \overline{U}_2)=0$, we have 
\[\mathrm{Ext}^1_{\mathbb{H}}(I(J_1,U_1), L(J_2, U_2)) \cong  0 .\]
This implies $\mathrm{Ext}^1_{\mathbb{H}}(L(J_1, U_1), L(J_2, U_2))=0$ by using a long exact sequence from the short exact sequence 
\[ 0 \rightarrow N(J_1,U_1) \rightarrow I(J_1,U_1) \rightarrow L(J_1,U_1) \rightarrow 0. \]

(4)(b) follows from Theorem \ref{thm jantan differential}(2).
\end{proof}

\begin{remark} \label{rmk relation with filt}
\begin{enumerate}
\item For Theorem \ref{thm first ext sum}(1) and (2), that is related to the second layer of the radical filtration of the corresponding standard module. For Theorem \ref{thm first ext sum}(4), that is related to the second layer of the Jantzen filtration. 
\item For Theorem \ref{thm first ext sum}(a), our approach only deals with under the assumption that $\mathrm{Ext}^1_{\mathbb{H}_J^{ss}}(\mathrm{Res}_{\mathbb{H}_J^{ss}}U_1, \mathrm{Res}_{\mathbb{H}_J^{ss}}U_2)=0$. Nevertheless, the assumption in Theorem \ref{thm first ext sum}(a) is satisfied by a range of examples. For example, if $\mathrm{Res}_{\mathbb{H}_J^{ss}}U_1$ is a discrete series, it is shown independently in \cite{Me}, \cite{OS} and \cite{Ch2} that $\mathrm{Ext}^1_{\mathbb{H}_J^{ss}}(\mathrm{Res}_{\mathbb{H}_J^{ss}}U_1, \mathrm{Res}_{\mathbb{H}_J^{ss}}U_2)=0$ ($U_2$ can be any arbitrary tempered modules). If $\mathrm{Res}_{\mathbb{H}_J^{ss}}U_1$ is elliptic and not a discrete series, then it can be checked from the result of \cite[Theorem 5.2]{OS2} that $\mathrm{Ext}_{\mathbb{H}_J^{ss}}^1(\mathrm{Res}_{\mathbb{H}_J^{ss}}U_1, \mathrm{Res}_{\mathbb{H}_J^{ss}}U_1) =0$ for a number of cases.
\end{enumerate}
\end{remark}

\subsection{Some computations on $\mathrm{Ext}$-groups}
In this section, we discuss some computations on $\mathrm{Ext}$-groups from results in this paper and \cite{Ch2} and some computations from \cite{Ci}.

\begin{example} \label{ex C3}
Here we give an example on computing $\mathrm{Ext}$-groups from our results. We shall assume a version of the Jantzen conjecture (see e.g. \cite[Conjecture 6.2.2]{BC})  to compute Jantzen filtrations from some Kazhdan-Lusztig polynomials \cite{Ci} (which uses \cite{Lu2, Lu3}). We mainly use to get information for the second layer of the Jantzen filtration.

Consider $\mathbb{H}$ of type $C_3$ as in \cite[Section 4.4]{Ci} and use the notation in that section. Denote by $I(4_b), I(3_{b,s})$, etc (resp. $L(4_a), L(3_{b,s})$, etc) the standard modules (resp. simple modules) associated to $4_a, 3_{b,s}$, etc respectively. Since $L(5_s)$ and $L(5_t)$ are discrete series, the $\mathrm{Ext}$-groups follow from \cite[Theorem 7.2]{Ch2}.

The standard module $I(4_b)$ satisfies the hypothesis in Theorem \ref{thm first ext sum}(4). Then by Theorem \ref{thm first ext sum}(4) and the Jantzen filtration (with assuming the truth of the conjecture), we have $\mathrm{Ext}^1_{\mathbb{H}}(L(4_a), L(4_b))=0$ and thus we have 
\[   \mathrm{Ext}^2_{\mathbb{H}}(L(4_a), L(4_a)) \cong \mathbb{C}, \quad \mathrm{Ext}^i_{\mathbb{H}}(L(4_a), L(4_a)) =0 \mbox{ for $i \neq 0,2$ } .
\]
We also have
\[  \mathrm{Ext}^i_{\mathbb{H}}(L(4_a), L(5_s)) =\mathrm{Ext}^i_{\mathbb{H}}(L(4_a), L(5_t))=0 \mbox{ for $i \neq 0$ } ,\]
and
\[   \mathrm{Ext}^1_{\mathbb{H}}(L(4_a), L(5_s)) \cong \mathrm{Ext}^1_{\mathbb{H}}(L(4_a), L(5_t)) \cong \mathbb{C}.
\]
We now turn to $3_{b,s}$. There are two possible radical filtration based on the Jantzen filtration. Suppose $\mathrm{rad}^1(I(3_{b,s})) \cong L(4_a) \oplus L(5_s) $. For such case, applying the $\mathrm{Hom}_{\mathbb{H}}(., L(5_s))$-functor and using Proposition \ref{prop ext lower langlands para}, we have $\mathrm{Ext}^2_{\mathbb{H}}(L(3_{b,s}), L(4_a)) \cong \mathrm{Ext}^2_{\mathbb{H}}(L(4_a)\oplus L(5_s), L(5_s)) \cong \mathbb{C}$. Now by applying the duality \cite[Theorem 4.15]{Ch2}, we have $\mathrm{Ext}^1_{\mathbb{H}}(L(4_a), L(4_b))\cong \mathbb{C}$ which contradicts to Theorem \ref{thm first ext sum} and the data in \cite[Section 4.4]{Ci}. Thus we can only have
\[ \mathrm{rad}^1(I(3_{b,s}))/\mathrm{rad}^2(I(3_{b,s})) \cong L(4_a), \quad  \mathrm{rad}^2(I(3_{b,s}))/\mathrm{rad}^3(I(3_{b,s})) \cong L(5_s).
\]
Then by standard homological algebra, we have
\begin{align} \label{eqn ext for pairs}
 \mathrm{Ext}^1_{\mathbb{H}}(L(3_{b,s}), L(4_a)) \cong \mathrm{Ext}^2_{\mathbb{H}}(L(3_{b,s}), L(5_t)) \cong \mathbb{C} 
\end{align}
and
\[ \mathrm{Ext}^i_{\mathbb{H}}(L(3_{b,s}), L(4_a)) = \mathrm{Ext}^j_{\mathbb{H}}(L(3_{b,s}), L(5_t))=\mathrm{Ext}^k_{\mathbb{H}}(L(3_{b,s}), L(5_s)) =0 \]
for all $i,j$ not as in (\ref{eqn ext for pairs}). By Theorem \ref{thm first ext sum} and Jantzen filtrations, we have 
\[ \mathrm{Ext}^1_{\mathbb{H}}(L(3_{b,s}), L(3_{b,s}))=0. \]
By using Theorem \ref{thm first ext sum}, we have
\[ \mathrm{Ext}^1_{\mathbb{H}}(L(3_{b,s}), L(3_{b,s})) \cong  \mathrm{Ext}^1_{\mathbb{H}}(L(3_{b,s}), L(4_b)) =0. \]

We now compute $\mathrm{Ext}^i_{\mathbb{H}}(L(3_{b,t}), L(3_{b,t}))$. By Theorem \ref{thm first ext sum}, 
\[\mathrm{Ext}^1_{\mathbb{H}}(L(3_{b,t}), L(3_{b,t})) \cong 0 ,\]
\[\mathrm{Ext}^2_{\mathbb{H}}(L(3_{b,t}), L(3_{b,t})) \cong \mathrm{Ext}^1_{\mathbb{H}}(L(3_{b,t}), L(2_{b})) \cong 0 ,\]
\[\mathrm{Ext}^3_{\mathbb{H}}(L(3_{b,t}), L(3_{b,t})) \cong \mathrm{Ext}^1_{\mathbb{H}}(L(3_{b,t}), L(2_{b})) \cong 0 .\]
By using the duality \cite[Theorem 4.15]{Ch2}, we also have 
\[ \mathrm{Ext}^1_{\mathbb{H}}(L(2_b), L(2_b))= \mathrm{Ext}^2_{\mathbb{H}}(L(2_b), L(2_b))=0 .\]
This is also compatible with the result of Theorem \ref{thm first ext sum}(4) and the Jantzen filtration of the first layer for $I(2_b)$.

Other pairs can be computed by similar manner, or using suitable duality to reduce to known $\mathrm{Ext}$-groups. We have also checked many cases that the resulting $\mathrm{Ext}$-groups and the structure of standard modules in terms of the (conjectured) Jantzen filtration are compatible. 

\end{example}



\begin{remark}
We remark that applying the Kazhdan-Lusztig polynomials and the Jantzen conjecture, one expect to obtain information for a generic vector in $V_{\mathrm{bad}}^{\bot}$ (which is also assumed in Example \ref{ex C3}). Thus it is easier to determine $\mathrm{Ext}^1_{\mathbb{H}}(L(J,U), L(J,U))$ for $|J|=|\Pi|-1$. For cases of larger $J$, one sometimes needs some more information.
\end{remark}

\begin{example} \label{ex yoneda product}
Here we show that one can obtain some other structure for the extension algebra from Theorem \ref{thm JF general stand}. We use the notation in Example \ref{ex C3}. We shall show that the Yoneda product 
\begin{align} \label{eqn yoneda prod} \mathrm{Ext}^1_{\mathbb{H}}(L(5_s), L(4_a)) \otimes \mathrm{Ext}^1_{\mathbb{H}}(L(4_a), L(5_s)) \rightarrow \mathrm{Ext}^2_{\mathbb{H}}(L(4_a), L(4_a)) 
\end{align}
is a non-zero map. 

Let $X= I(4_a)/L(5_t)$. We consider the short exact sequence:
\[   0\rightarrow L(5_s) \rightarrow X \rightarrow L(4_a) \rightarrow 0 .
\]
By applying the $\mathrm{Hom}_{\mathbb{H}}(., L(4_a))$ functor, we have the long exact sequence
\[ \ldots \rightarrow \mathrm{Ext}^1_{\mathbb{H}}(L(4_a), L(4_a)) \rightarrow \mathrm{Ext}^1_{\mathbb{H}}(X, L(4_a)) \rightarrow \mathrm{Ext}^1_{\mathbb{H}}(L(5_s), L(4_a)) \stackrel{\partial}{\rightarrow} \mathrm{Ext}^2_{\mathbb{H}}(L(4_a), L(4_a)) \rightarrow \ldots 
\]
The map $\partial$ coincides with the Yoneda product in (\ref{eqn yoneda prod}). Hence if $\partial$ is zero, we have $\mathrm{Ext}^1_{\mathbb{H}}(X, L(4_a)) \cong \mathbb{C}$. Then by comparing dimensions, the natural surjective map $I(4_a) \rightarrow X$ induces an isomorphism $\mathrm{Ext}^1_{\mathbb{H}}(X, L(4_a)) \cong \mathrm{Ext}^1_{\mathbb{H}}(I(4_a), L(4_a))$. Now considering the Yoneda construction of the modules for $\mathrm{Ext}^1_{\mathbb{H}}(X, L(4_a))$ and $\mathrm{Ext}^1_{\mathbb{H}}(I(4_a), L(4_a))$, we obtain a module $L(J,U^{2,\eta^{\vee}})$ which does not have a unique simple module. This gives a contradiction and hence we have the Yoneda product to be non-zero. Here $(J,U) \in \Xi_L$ such that $L(J,U) =L(4_a)$. 
\end{example}




\section{Appendix A: Jantzen filtration for an example of type $B_2$ }

We keep using the notation in Example \ref{ss remark filt}. Fix a basis $\left\{ u_1, u_2 \right\}$ of $U$ such that the action of $S(V)$ on $\phi(U_{\mathbf t\nu^{\vee}})$ in matrix form with respect to the basis $\left\{ \phi(u_1), \phi(u_2) \right\}$ is as follows:
\[ \alpha =  \begin{bmatrix} 0 & 0 \\ 1 & 0 \end{bmatrix}  \]
 and 
\[\beta = \begin{bmatrix} -(2+2\mathbf t) & 0 \\ -1 & -(2+2\mathbf t) \end{bmatrix} .\]
We also have
\[ (2\alpha + \beta)^2-4= 4\begin{bmatrix} \mathbf t(2+\mathbf t)& 0 \\ -(1+\mathbf t) & \mathbf t(2+\mathbf t) \end{bmatrix} 
\]
and
\[ (\alpha+\beta)^2-4 = 4\begin{bmatrix}\mathbf t(2+\mathbf t) & 0 \\ 0 &\mathbf t(2+\mathbf t)  \end{bmatrix} 
\]
and
\[ \beta^2-4 = 4\begin{bmatrix} \mathbf t(2+\mathbf t) & 0 \\ (1+\mathbf t) & \mathbf t(2+\mathbf t) \end{bmatrix} .
\]

We now regard $\phi(u_1)$ and $\phi(u_2)$ as elements in $\phi(U_{\mathbf t\nu^{\vee}})$. Set $u^1_{\mathbf t}= \beta^{-1} (\alpha+\beta)^{-1} (2\alpha+\beta)^{-1} \phi(u_1) \in \phi(U_{\mathbf t\nu^{\vee}})$ and $u^2_{\mathbf t}=\beta^{-1} (\alpha+\beta)^{-1} (2\alpha+\beta)^{-1} \phi(u_2) \in \phi(U_{\mathbf t\nu^{\vee}})$. Note that $1 \otimes u^1_{\mathbf t}$ and $1 \otimes u^2_{\mathbf t}$ are holomorphic and both of them specialized at $\mathbf t=0$ are nonzero. 

Now we compute the image of $\Delta_{\mathbf t\nu^{\vee}}$ of the following elements.
\begin{align*}
\Delta_{\mathbf t\nu^{\vee}}(1 \otimes u_1) &= \widetilde{\tau}_{s_{\beta}s_{\alpha}s_{\beta}} \otimes u^1_{\mathbf t} \\
\Delta_{\mathbf t\nu^{\vee}}(1 \otimes u_2) &= \widetilde{\tau}_{s_{\beta}s_{\alpha}s_{\beta}} \otimes  u^2_{\mathbf t} \\
\Delta_{\mathbf t\nu^{\vee}}(\widetilde{\tau}_{s_{\beta}} \otimes u_1) &= 4\mathbf t(2+\mathbf t)\widetilde{\tau}_{s_{\alpha}s_{\beta}} \otimes u^1_{\mathbf t}-4(1+\mathbf t)\widetilde{\tau}{s_{\alpha}s_{\beta}} \otimes  u^2_{\mathbf t}  \\
\Delta_{\mathbf t\nu^{\vee}}(\widetilde{\tau}_{s_{\beta}} \otimes u_2) &= 4\mathbf t(1+\mathbf t)\widetilde{\tau}_{s_{\alpha}s_{\beta}} \otimes  u^2_{\mathbf t}  \\
\Delta_{\mathbf t\nu^{\vee}}(\widetilde{\tau}_{s_{\alpha}s_{\beta}} \otimes u_1) &= 16\mathbf t^2(2+\mathbf t)^2\widetilde{\tau}_{s_{\beta}} \otimes u^1_{\mathbf t}-16\mathbf t(1+\mathbf t)(2+\mathbf t)\widetilde{\tau}_{s_{\beta}} \otimes  u^2_{\mathbf t}  \\
\Delta_{\mathbf t\nu^{\vee}}(\widetilde{\tau}_{s_{\alpha}s_{\beta}} \otimes u_2) &= 16\mathbf t^2(2+\mathbf t)^2\widetilde{\tau}_{s_{\beta}} \otimes u^1_{\mathbf t}  \\
\Delta_{\mathbf t\nu^{\vee}}(\widetilde{\tau}_{s_{\beta}s_{\alpha}s_{\beta}} \otimes u_1) &= 64\mathbf t^3(2+\mathbf t)^3 \otimes u_{\mathbf t}^1 \\
\Delta_{\mathbf t\nu^{\vee}}(\widetilde{\tau}_{s_{\beta}s_{\alpha}s_{\beta}} \otimes u_2) &= 64\mathbf t^3(2+\mathbf t)^3 \otimes u_{\mathbf t}^2 
\end{align*}

Note that the image of $1 \otimes u_1,\ 1 \otimes u_2,\ \widetilde{\tau}_{s_{\beta}} \otimes u_1$ span $\mathrm{JF}^0_{\nu^{\vee}}/\mathrm{JF}^1_{\nu^{\vee}}$. Similarly the image of 
$\widetilde{\tau}_{s_{\alpha}s_{\beta}} \otimes u_1$ spans $\mathrm{JF}^1_{\nu^{\vee}}/\mathrm{JF}^2_{\nu^{\vee}}$. The image of $\widetilde{\tau}_{s_{\beta}} \otimes u_2+\mathbf t\widetilde{\tau}_{s_{\beta}} \otimes u_1$ spans $\mathrm{JF}^2_{\nu^{\vee}}/\mathrm{JF}^3_{\nu^{\vee}}$. The image of $\widetilde{\tau}_{s_{\alpha}s_{\beta}}  \otimes u_2+2\mathbf t\widetilde{\tau}_{s_{\alpha}s_{\beta}} \otimes u_1,\ \widetilde{\tau}_{s_{\beta}s_{\alpha}s_{\beta}} \otimes u_1,\ \widetilde{\tau}_{s_{\beta}s_{\alpha}s_{\beta}} \otimes u_2$ spans $\mathrm{JF}^3_{\nu^{\vee}}/\mathrm{JF}^4_{\nu^{\vee}}$.

\section{Appendix B: Proof of Proposition \ref{prop wv limit form}}

In the following proofs, we assume the reader is familiar with the standard properties of Bruhat-Chevalley ordering (see e.g. \cite{Hu}). Some facts are used without mentioning explicitly. Define $l: W \rightarrow \mathbb{Z}_{\geq 0}$ be the length function on $W$. 

\begin{lemma} \label{lem various l case a}
Let $w \in W^J$. Let $s_{\alpha}$ be a simple reflection. Then we have one of the following cases:
\begin{enumerate}
\item $l(s_{\alpha}w)=l(w)+1$ and $s_{\alpha}w \in W^J$. In this case, $w^{-1}(\alpha) \in R \setminus R_J$. Furthermore, $(t_{s_{\alpha}}\alpha-k_{\alpha})\tau_w=\tau_{s_{\alpha}w}w^{-1}(\alpha)$ and $t_{s_{\alpha}} \tau_w= \tau_{s_{\alpha}w}+k_{\alpha} \tau_w w^{-1}(\alpha)^{-1}$.  
\item $l(s_{\alpha}w)=l(w)+1$ and $s_{\alpha}w \notin W^J$. In this case, $s_{\alpha}w=ws_{\alpha'}$ for some $\alpha' \in J$. Furthermore $t_{s_{\alpha}}\widetilde{\tau}_w =\widetilde{\tau}_w t_{s_{\alpha'}}$ and $(t_{s_{\alpha}}\alpha-k_{\alpha})\widetilde{\tau}_w=\widetilde{\tau}_{w}(t_{s_{\alpha'}}\alpha'-k_{\alpha'})$.  
\item $l(s_{\alpha}w)=l(w)-1$. In this case, $s_{\alpha}w \in W^J$ and $w^{-1}(\alpha) \in R \setminus R_J$. Furthermore, $(t_{s_{\alpha}}\alpha-k_{\alpha})\tau_w=\tau_{s_{\alpha}w} (w^{-1}(\alpha)^2-k_{\alpha}^2)w^{-1}(\alpha)^{-1}$ and $t_{s_{\alpha}} \tau_w= \tau_{s_{\alpha}w}(w^{-1}(\alpha)^2-k_{\alpha}^2)w^{-1}(\alpha)^{-2}+k_{\alpha} \widetilde{\tau}_w w^{-1}(\alpha)^{-1}$.  
\end{enumerate}
\end{lemma}

\begin{proof}
(1) follows from definitions (and some details are similar to (2) below). 

For (2), suppose $s_{\alpha}w \notin W^J$. Let $R(w) = \left\{ \beta \in R^+ : w(\beta)<0 \right\}$. Since $R(w) \subset R(s_{\alpha}w)$ and $R(w) \cap R_J=\emptyset$, $|R(s_{\alpha}) \cap R_J| \leq 1$. The condition that $s_{\alpha}w \notin W^J$ implies that $|R(s_{\alpha}) \cap R_J| = 1$. By a unique factorization of an element into the product of an element in $W^J$ and an element in $W_J$, we have $s_{\alpha}w=ws_{\alpha'}$ for some $\alpha' \in \Pi$. The assertion that $(t_{s_{\alpha}}\alpha-k_{\alpha})\widetilde{\tau}_w=\widetilde{\tau}_{w}(t_{s_{\alpha'}}\alpha'-k_{\alpha'})$ follows again from  the proof of \cite[Proposition 2.5]{KR} (c.f. Proposition \ref{prop indep cho base}). Since $w^{-1}(\alpha)=\alpha'$, we also have $\alpha \widetilde{\tau}_w=\widetilde{\tau}_w\alpha'$ (see Lemma \ref{lem basic comm lem}). Then one can verify that the action of $t_{s_{\alpha}}\widetilde{\tau}_w$ and $\widetilde{\tau}_w t_{s_{\alpha'}}$ is the same on a faithful $\mathbb{H}$-module described in the proof of \cite[Proposition 2.8(e)]{KR}. Hence $t_{s_{\alpha}}\widetilde{\tau}_w=\widetilde{\tau}_w t_{s_{\alpha'}}$.

For (3), since $l(s_{\alpha}w)=l(w)-1$, $R(s_{\alpha}w) \subset R(w)$ and so $s_{\alpha}w \in W^J$. 
\end{proof}

\noindent
{\it Proof of Proposition \ref{prop wv limit form}}

Let $w \in W^J$. Let
\[  \Lambda(w, \gamma^{\vee}) = \left\{ \lambda^{\vee} \in \mathrm{Wgt}(U): w(\lambda^{\vee})=\gamma^{\vee} \mbox{ for some $\lambda^{\vee}\in \mathrm{Wgt}(U) $}\right\}.
\]


For each $\lambda^{\vee} \in \mathrm{Wgt}(U)$, let $u_1, \ldots, u_{r_{\lambda}^{\vee}}$ form a basis of the generalized $\lambda^{\vee}$-weight space and regard those elements in $1 \otimes U \subset U_{\mathbf a}$.  

For each $w_1 \in W^J$ and $\lambda^{\vee}_1 \in \mathrm{Wgt}(U)$ such that$\gamma^{\vee}=w_1(\lambda^{\vee}_1)$, and for each $k=1, \ldots, r_{\lambda_1^{\vee}}$, we shall construct vectors $x_{\mathbf a} \in \widetilde{U}$ of the form
\begin{align} \label{eqn phi form}  x_{\mathbf a}= \tau_{w_1} p_{w_1,\lambda_1^{\vee}, k} \otimes u_{\lambda_1^{\vee}, k} +\sum_{w \in W(J,U, \gamma^{\vee}), w_1 > w} \sum_{\lambda^{\vee} \in \Lambda(w,\gamma^{\vee})} \sum_{i=1}^{r_{\lambda^{\vee}}} \tau_w q_{w,\lambda^{\vee}, i} \otimes u_{\lambda^{\vee},i}
\end{align}
with 
\begin{enumerate}
\item[(i)] $x_{\mathbf a}$ is holomorphic;
\item[(ii)] $p_{w_1,\lambda_1^{\vee},k}\in S(V) \subset \mathcal O(J)$;
\item[(iii)] $x_{\mathbf a}|_{\mathbf a=0}$ has weight $\gamma^{\vee}$;
\item[(iv)] $\lambda^{\vee}_1(p_{w_1,\lambda_1^{\vee},k}) \neq 0$.
\end{enumerate}

 By the definition of $|_{\mathbf a=0}$ and property (iv) above, we see that $x$ is a non-zero scalar multiple of an element of the form 
\begin{align*} t_{w_1} \otimes u_{\lambda_1^{\vee},k}+\sum_{w \in W^J, w_1>w} t_w \otimes u_w 
\end{align*}
for some $u_w \in U$. 

Fix $\lambda^{\vee} \in \mathrm{Wgt}(U)$. For $w=1 \in W^J$, there is nothing to prove. Let $w_1 \in W^J$ and with $w' \neq 1$ and let $w_1=s_{\alpha_1}\ldots s_{\alpha_r}$ be a reduced expression of $w_1$. Then $w_2=s_{\alpha_2} \ldots s_{\alpha_r}$, which is also in $W^J$ by definitions. By our inductive construction, we can assume there exists an element $x_{\mathbf a}$ the form (\ref{eqn phi form}) starting with the term $\tau_{w_2} p_{w_2,\lambda_1^{\vee}, k} \otimes u_{\lambda_1^{\vee}, k}$  satisfying properties (i) to (iii). 

Here we divide into few cases. Before that, set $\alpha=\alpha_1$ for simplicity of notations. For the first case, suppose $s_{\alpha}(w_2(\lambda^{\vee}))=w_2(\lambda^{\vee})$, equivalently $w_2(\lambda^{\vee})(\alpha)=\alpha$. In this case, set $\widetilde{x}_{\mathbf a}=t_{s_{\alpha}}x_{\mathbf a}$. Note that $t_{s_{\alpha}}.\widetilde{x}_{\mathbf a}$ is also holomorphic and $(t_{s_{\alpha}} \widetilde{x}_{\mathbf a})|_{\mathbf a=0}=t_{s_{\alpha}}(\widetilde{x}|_{\mathbf a=0})$. Set $\gamma^{\vee}=w_2(\lambda^{\vee})$. We also have 
\begin{align}
(v-\gamma^{\vee}(v))t_{s_{\alpha}} (\widetilde{x}_{\mathbf a})|_{\mathbf a=0}&=t_{s_{\alpha}}(s_{\alpha}(v)-\gamma^{\vee}(v))(\widetilde{x}_{\mathbf a})|_{\mathbf a=0}+\alpha^{\vee}(v)(\widetilde{x}_{\mathbf a})|_{\mathbf a=0} \\
                                                      &=t_{s_{\alpha}}(s_{\alpha}(v)-\gamma^{\vee}(s_{\alpha}(v)))(\widetilde{x}_{\mathbf a})|_{\mathbf a=0}+\alpha^{\vee}(v)(\widetilde{x}_{\mathbf a})|_{\mathbf a=0} \quad \mbox{ (by $s_{\alpha}(\gamma^{\vee})=\gamma^{\vee}$)} 
\end{align}
and so $(v-\gamma^{\vee}(v))^lt_{s_{\alpha}} (\widetilde{x}_{\mathbf a}|_{\mathbf a=0})=0$ for sufficiently large $l$. Thus $t_{s_{\alpha}}(\widetilde{x}_{\mathbf a})|_{\mathbf a=0}$ is a generalized weight vector with the weight $w_1(\lambda^{\vee})=w_2(\lambda^{\vee})$. This shows that $\widetilde{x}_{\mathbf a}$ satisfies property (iii).

We now rewrite $\widetilde{x}_{\mathbf a}$ to the form as in (\ref{eqn phi form}). We consider the leading term $t_{s_{\alpha}} \tau_{w_2} p_{w_2,\lambda_1^{\vee}, k} \otimes u_{\lambda_1^{\vee}, k}$ and write as
\begin{align*}
 & t_{s_{\alpha}} \tau_{w_2} p_{w_2,\lambda_1^{\vee}, k} \otimes u_{\lambda_1^{\vee}, k} \\
=& t_{s_{\alpha}}\alpha \tau_{w_2} (w_2^{-1}(\alpha)^{-1}p_{w_2,\lambda_1^{\vee}, k}) \otimes u_{\lambda_1^{\vee}, k} \quad \mbox{ ($s_{\alpha}w_2 \in W^J$ implies $w_2^{-1}(\alpha)\in R \setminus R_J$)} \\
=& (t_{s_{\alpha}}\alpha - k_{\alpha}) \tau_{w_2} (w_2^{-1}(\alpha)^{-1}p_{w_2,\lambda_1^{\vee}, k}) \otimes u_{\lambda_1^{\vee}, k} +k_{\alpha} \tau_{w_2} (w_2^{-1}(\alpha)^{-1}p_{w_2,\lambda_1^{\vee}, k}) \otimes u_{\lambda_1^{\vee}, k} \\
=&  \tau_{s_{\alpha_1}w_2} p_{w_2,\lambda_1^{\vee}, k} \otimes u_{\lambda_1^{\vee}, k} +k_{\alpha} \tau_{w_2} (w_2^{-1}(\alpha)^{-1}p_{w_2,\lambda_1^{\vee}, k}) \otimes u_{\lambda_1^{\vee}, k} 
\end{align*}

Other terms can be rewritten in a similar fashion with the use of Lemma \ref{lem various l case a}. We remark that if the term falls in the case of Lemma \ref{lem various l case a}(2), the algebra structure from Lemma \ref{lem rat sturcture} is also needed. 

Recall that from our inductive construction, $p_{w_2,\lambda_1^{\vee}, k} \in S(V)$ and $\lambda^{\vee}_1(p_{w_2,\lambda_1^{\vee},k}) \neq 0$. By looking at the leading term of $t_{s_{\alpha}}\widetilde{x}_{\mathbf a}$ (in the form as in (\ref{eqn phi form})), we see that $t_{s_{\alpha}}\widetilde{x}$ satisfies properties (ii) and (iv). Hence $t_{s_{\alpha}}\widetilde{x}_{\mathbf a}$ gives the desired element in the case of $s_{\alpha_1}(w_2(\lambda^{\vee}))=w_2(\lambda^{\vee})$. This completes the verification for the case $s_{\alpha}(w_2(\lambda^{\vee}))=w_2(\lambda^{\vee})$.

We now consider the case $s_{\alpha}(w_2(\lambda^{\vee}))\neq w_2(\lambda^{\vee})$. In this case, let $\widetilde{x}_{\mathbf a}=(t_{s_{\alpha}}\alpha-k_{\alpha})x$. We have to rewrite $\widetilde{x}_{\mathbf a}$ to the form (\ref{eqn phi form}). Again we only do it for the leading terms and other terms can be rewritten similarly with the use of Lemma \ref{lem various l case a}.
\begin{align} \label{eqn original form}
 & (t_{s_{\alpha}}\alpha-k_{\alpha}) \tau_{w_2} p_{w_2,\lambda_1^{\vee}, k} \otimes u_{\lambda_1^{\vee}, k} \\
=& \tau_{s_{\alpha}w_2} (w_2^{-1}(\alpha)p_{w_2,\lambda_1^{\vee}, k}) \otimes u_{\lambda_1^{\vee}, k} \quad \mbox{ (by Lemma \ref{lem various l case a} (1))} 
\end{align}

To check property (i), one can use Lemma \ref{lem basic comm lem}. Properties (ii) and (iv) follow from the facts that $(w_2^{-1}(\alpha)p_{w_2,\lambda_1^{\vee}, k}) \in S(V)$ and $\lambda_1^{\vee}(w_2^{-1}(\alpha)p_{w_2,\lambda_1^{\vee}, k})=(w_2(\lambda^{\vee})(\alpha))\lambda^{\vee}(p_{w_2,\lambda_1^{\vee}, k}))\neq 0$. Here $w_2(\lambda^{\vee})(\alpha) \neq 0$ because of our assumption that $s_{\alpha}(w_2(\lambda^{\vee}))\neq w_2(\lambda^{\vee})$.

From (\ref{eqn phi form}), we see that the weight vectors we constructed are linearly independent. By counting the dimension, those weight vectors form a basis for $\mathbb{H} \otimes_{\mathbb{H}_J} U$. Then using the property (ii), we obtain the statement.


\begin{thebibliography}{AFMO}

\bibitem[ALTV]{ALTV} J. Adams, M. van Leeuwen, P. Trapa and D. Vogan, {\it Unitary representations of real reductive groups}, arXiv:1212.2192.


\bibitem[AS]{AS} T. Arakawa and T. Suzuki, {\it Duality  between $\mathfrak{sl}(n,\mathbb{C})$ and  the  degenerate  affine  Hecke algebra}, J. Algebra {\bf 209} (1998), 288-304.

\bibitem[Ba]{Ba} D. Barbasch, {\it Filtrations on Verma modules}, Ann. Sci. l'\'Ecole Norm. Sup. (4) {\bf 16} (1983), 489-494.
 

\bibitem[BC]{BC} D. Barbasch and D. Ciubotaru, {\it Hermitian forms for affine Hecke algebras}, arXiv:1312.3316v1 [math.RT] (2015).

\bibitem[BC2]{BC2} D. Barbasch and D. Ciubotaru, {\it Star operations for affine Hecke algebras}, arXiv:1504.04361 [math.RT] (2015).


\bibitem[BM]{BM} D. Barbasch and A. Moy, {\it A unitarity criterion for p-adic groups}, Invent. Math. {\bf 98} (1989), 19-37.%


\bibitem[BM2]{BM2} D. Barbasch and A. Moy, {\it Unitary spherical spectrum for p-adic classical groups}, Acta Appl. Math. {\bf 44} (1996), no. 1-2, 3-37.

\bibitem[BB]{BB} A. Beilinson and I. N. Bernstein, {\it A proof of Jantzen conjecture}, Adv. in Soviet Math. {\bf 16}, Part 1 (1993), 1-50.

\bibitem[Be]{Be} D. J. Benson, {\it Representations and cohomology. I: Basic representation theory of finite groups and associative algebras}, 2nd ed., Cambridge Studies in Advanced Mathematics, {\bf 30}, Cambridge University Press, Cambridge, 1998.

\bibitem[BW]{BW} A. Borel and N. Wallach, {\it Continuous cohomology, discrete subgroups, and representations of reductive groups}, {\bf 94} (1980), Princeton University Press (Princeton, NJ).





\bibitem[Ch]{Ch} K. Y. Chan, {\it Extensions of graded affine Hecke algebra modules}, PhD thesis, University of Utah, 2014.

\bibitem[Ch2]{Ch2} K. Y. Chan, {\it Duality for Ext-groups and extensions of discrete series for graded Hecke algebras}, Adv. Math. 294 (2016), 410-453.

\bibitem[CG]{CG} N. Chriss and V. Ginzburg, Representation Theory and Complex Geometry, Birkh\"auser (Boston), 1997.

\bibitem[Ci]{Ci} D. Ciubotaru, {\it Multiplicity matrices for the graded affine Hecke algebra}, J. Algebra 320 (2008), 3950-3983.












\bibitem[DO]{DO2} P. Delorme and E. M. Opdam, {\it Analytic $R$-groups for affine Hecke algebras}, J. Reine Angew. Math., {\bf 658} (2011), 133-172.

\bibitem[Ev]{Ev} S. Evens, {\it The Langlands classification for graded Hecke algebras}, Proc. Amer. Math. Soc. {\bf 124} (1996), no. 4, 1285-1290.

\bibitem[EM]{EM} S. Evens and I. Mirkovi\'c, {\it Fourier transform and the Iwahori-Matsumoto involution}, Duke Math. J. {\bf 86} (1997), no. 3, 435-464.





\bibitem[Hu]{Hu} J. Humphreys, {\it Reflection groups and Coxeter groups}, Cambridge Univ. Press, 1990.



\bibitem[Ir]{Ir}  R. S. Irving,  {\it Projective Modules in the Category $\mathcal O_S$. Loewy Series}, Trans. Amer. Math. Soc., {\bf 291}, 1985, 733-754.



\bibitem[Ja]{Ja} J. C. Jantzen,
{\it Modulen  mit einem h\"ochsten Gewicht}, Lect. Notes in Math. {\bf 750}, Springer: Berlin-Heidelberg-New York (1979).

\bibitem[Ka]{Ka} S. Kato, {\it A homological study of Green polynomials},  to appear in Ann. Sci. l'\'Ecole Norm. Sup., arXiv:1111.4640.

\bibitem[KL]{KL} D. Kazhdan and G. Lusztig, {\it Proof for the Deligne-Langlands conjecture for Hecke algebras}, Invent. Math.,  {\bf 87} (1987), 153-215.




\bibitem[Kn]{Kn0} A. W. Knapp, {\it Representation theory of semisimple Lie groups: an overview based on examples}, Princeton University Press, 1986.

\bibitem[KS]{KS} A. W. Knapp and E. M. Stein, {\it Intertwining operators for semisimple groups II}, Invent. Math. {\bf 60} (1980), 9-84.




\bibitem[KR]{KR} C. Kriloff and A. Ram, {\it Representations of graded Hecke algebras},  Represent. Theory {\bf 6} (2002), 31-69.

\bibitem[La]{La} R. Langlands, {\it On the classification of irreducible representations of real algebraic groups}, in Representation theory and harmonic analysis on semisimple Lie groups, 101-170, Math. Surveys Monogr., {\bf 31}, Amer. Math. Soc., Providence, RI, 1989.

\bibitem[Lu]{Lu0} G. Lusztig, {\it Cuspidal local systems and graded Hecke algebras I}, Publ. Math. IH\'ES, {\bf 67} (1988), 145-202.

\bibitem[Lu2]{Lu} G. Lusztig, {\it Affine Hecke algebras and their graded versions}, J. Amer. Math. Soc. {\bf 2} (1989), 599-635.

\bibitem[Lu3]{Lu2} G. Lusztig, {\it Study of perverse sheaves arising from graded Lie algebras}, Adv.Math. {\bf 112} (1995), 147-217.

\bibitem[Lu4]{Lu3} G. Lusztig, {\it Graded Lie algebras and intersection cohomology}, Representation theory of algebraic groups and quantum groups, Progr. Math. 284 (2010), 
Birkh\"auser/Springer, New York, 191-244.


\bibitem[Me]{Me} R. Meyer, {\it  Homological algebra for Schwartz algebras of reductive p-adic groups}, Noncommutative geometry and number theory, Aspects of Mathematics E37 (2006), Vieweg Verlag, Wiesbaden, 263-300.




\bibitem[Op]{Op} E. M. Opdam, {\it Harmonic analysis for certain representations of graded Hecke algebras}, Acta. Math. {\bf 175} (1995), no. 1, 75-121. 

\bibitem[Op2]{Op2} E. M. Opdam, {\it Central support of the Plancherel measure of an affine Hecke algebra}, Mosc. Math. J., {\bf 7} (2007), 723-741.



\bibitem[OS]{OS} E. M. Opdam and M. Solleveld, {\it Homological algebra for affine Hecke algebras}, Adv. in Math. {\bf 220} (2009), 1549-1601.


\bibitem[OS2]{OS2} E. M. Opdam and M. Solleveld, {\it Extensions of tempered representations}, Geometric And Functional Analysis {\bf 23} (2013), 664-714.


\bibitem[Re]{Re0} M. Reeder, {\it Nonstandard intertwining operators and the structure of unramified principal series representations}, Forum Math. {\bf 9} (1997), 457-516.

\bibitem[Ro]{Ro}  J. Rogawski, {\it On modules over the Hecke algebra of a p-adic group}, Invent. Math. {\bf 79} (1985), 443-465.


\bibitem[Su]{Su} T. Suzuki, {\it Rogawski's conjecture on the Jantzen filtration for the degenerate affine Hecke algebra of type A}, Represent. Theory, 2 (1998), 393-409.


\bibitem[So]{So} M. Solleveld, {\it Parabolically induced representations of graded Hecke algebras},  Algebras and Representation Theory {\bf 152} (2012), 233-271. 

\bibitem[Vo]{Vo} D. Vogan, {\it Irreducible characters of semisimple Lie groups. II. The Kazhdan-Lusztig conjectures}, Duke Math. J. {\bf 46} (1979), no. 4, 805-859.

\bibitem[We]{We} C. Weibel, {\it An introduction to homological algebra}, Cambridge Studies in Advanced Mathematics 38, Cambridge University Press (1994).















\end{thebibliography}
\end{document}